\documentclass{siamltex}
\usepackage{amsfonts,amsmath,amssymb}
\usepackage{mathrsfs,mathtools}
\usepackage{enumerate}
\usepackage{xspace,MACROS/mydef}
\usepackage{hyperref}
\usepackage{esint}
\usepackage{graphicx}
\DeclareMathAlphabet{\mathpzc}{OT1}{pzc}{m}{it}

\usepackage{color}
\newcommand{\EO}[1]{\textcolor{black}{#1}}

\newtheorem{remark}[theorem]{Remark}
\numberwithin{equation}{section}

\newcommand{\calL}{{\mathcal L}}

\title{A FEM for an optimal control problem of fractional powers of elliptic operators
\thanks{EO has been supported in part by NSF grants DMS-1109325 and DMS-1411808.}}

\author{Harbir Antil\thanks{Department of Mathematical Sciences,
George Mason University, Fairfax, VA 22030, USA. \texttt{hantil@gmu.edu}},
\and
Enrique Ot\'arola\thanks{Department of Mathematics, University of Maryland,
College Park, MD 20742, USA and Department of Mathematical Sciences,
George Mason University, Fairfax, VA 22030, USA. \texttt{kike@math.umd.edu}}}

\date{Draft version of \today.}

\begin{document}

\maketitle
\begin{abstract}
We study solution techniques for a linear-quadratic optimal control problem involving fractional powers of elliptic operators. These fractional operators can be realized as the Dirichlet-to-Neumann map for a nonuniformly elliptic problem posed on a semi-infinite cylinder in one more spatial dimension. Thus, we consider an equivalent formulation with a nonuniformly elliptic operator as state equation. The rapid decay of the solution to this problem suggests a truncation that is suitable for numerical approximation. We discretize the proposed truncated state equation using first degree tensor product finite elements on anisotropic meshes. For the control problem we analyze two approaches: one that is semi-discrete based on the so-called variational approach, where the control is not discretized, and the other one is fully discrete via the discretization of the control by piecewise constant functions. For both approaches, we derive a priori error estimates with respect to the degrees of freedom. Numerical experiments validate the derived error estimates and reveal a competitive performance of anisotropic over quasi-uniform
refinement.
\end{abstract}



\section{Introduction}
\label{sec:introduccion}
We are interested in the design and analysis of numerical schemes for a linear-quadratic optimal control problem involving fractional powers of elliptic operators. To be precise, let $\Omega$ be an open and bounded domain of $\R^n$ ($n\ge1$), with boundary $\partial\Omega$.  Given $s \in (0,1)$, and a desired state $\usf_d: \Omega \rightarrow \mathbb{R}$, we define
\begin{equation}
\label{functional}
J(\usf,\zsf)= \frac{1}{2}\| \usf - \usf_{d} \|^2_{L^2(\Omega)} + 
\frac{\mu}{2} \| \zsf\|^2_{L^2(\Omega)},
\end{equation}
where $\mu > 0$ is the so-called regularization parameter. We shall be concerned with the following optimal control problem: Find
\begin{equation}
\label{Jintro}
 \text{min }J(\usf,\zsf),
\end{equation}
subject to the \emph{fractional state equation}
\begin{equation}
\label{fractional}
\mathcal{L}^s \usf = \zsf  \text{ in } \Omega, \qquad \usf = 0   \text{ on } \partial \Omega, \\
\end{equation}
and the \emph{control constraints}
\begin{equation}
 \label{cc}
\asf(x') \leq \zsf(x') \leq \bsf(x') \quad\textrm{a.e~~} x' \in \Omega . 
\end{equation}
The functions $\asf$ and $\bsf$ both belong to $L^2(\Omega)$ and satisfy the property $\asf(x') \leq \bsf(x')$ for almost every $x' \in \Omega$. The operator $\calLs$, $s \in (0,1)$, is a fractional power of the second order, symmetric and uniformly elliptic operator $\mathcal{L}$, supplemented with homogeneous Dirichlet boundary conditions:
\begin{equation}
\label{second_order}
 \mathcal{L} w = - \DIV_{x'} (A \nabla_{x'} w ) + c w,
\end{equation}
where $0 \leq c \in L^\infty(\Omega)$ and $A \in C^{0,1}(\Omega,\GL(n,\R))$ is symmetric and positive definite. For convenience, we will refer to the optimal control problem defined by \eqref{Jintro}-\eqref{cc} as the \emph{fractional optimal control problem}; see \S\ref{sub:control_fractional} for a precise definition.
 
Concerning applications, \emph{fractional diffusion} has received a great deal of attention in diverse areas of science and engineering. For instance, mechanics \cite{atanackovic2014fractional}, biophysics \cite{bio}, turbulence \cite{wow}, image processing \cite{GH:14}, peridynamics \cite{HB:10}, nonlocal electrostatics \cite{ICH} and finance \cite{MR2064019}. In many of these applications, control problems arise naturally.

One of the main difficulties in the study of problem \eqref{fractional} is the nonlocality of the fractional operator $\mathcal{L}^s$ (see \cite{CT:10,CS:07,CDDS:11,NOS,ST:10}). A possible approach to overcome this nonlocality property is given by the seminal result of Caffarelli 
and Silvestre in $\mathbb{R}^n$ \cite{CS:07} and its extensions to both bounded domains \cite{CT:10,CDDS:11} and a general class of elliptic operators \cite{ST:10}.
Fractional powers of $\mathcal{L}$ can be realized as an operator that maps a Dirichlet boundary condition to a Neumann condition via an extension problem 
on $\C = \Omega \times (0,\infty)$. This extension leads to the following mixed boundary value problem:
\begin{equation}
\label{alpha_harm_L}
\mathcal{L}\ue - \frac{\alpha}{y}\partial_{y}\ue - \partial_{yy}\ue = 0 \text{ in } \C, 
\quad
\ue = 0 \text{ on } \partial_L \C, 
\quad
\frac{ \partial \ue }{\partial \nu^\alpha} = d_s \zsf \text{ on } \Omega \times \{0\},
\end{equation}
where $\partial_L \C= \partial \Omega \times [0,\infty)$ is the lateral boundary of $\C$, $\alpha = 1-2s \in (-1,1)$, $d_s = 2^{\alpha}\Gamma(1-s)/\Gamma(s)$ and 
the conormal exterior derivative of $\ue$ at $\Omega \times \{ 0 \}$ is
\begin{equation}
\label{def:lf}
\frac{\partial \ue}{\partial \nu^\alpha} = -\lim_{y \rightarrow 0^+} y^\alpha \ue_y.
\end{equation}

We will call $y$ the \emph{extended variable} and the dimension $n+1$ in $\R_+^{n+1}$ the \emph{extended dimension} of problem \eqref{alpha_harm_L}. The limit in \eqref{def:lf} must be understood in the distributional sense; see \cite{CS:07,ST:10}. As noted in \cite{CT:10,CS:07,CDDS:11,ST:10}, we can relate the fractional powers of the operator $\calL$ with the Dirichlet-to-Neumann map of problem 
\eqref{alpha_harm_L}:
$
  d_s \calLs u = \tfrac{\partial \ue}{\partial \nu^\alpha }
$
in $\Omega$. Notice that the differential operator in \eqref{alpha_harm_L} is
$
  -\DIV \left( y^{\alpha} \mathbf{A} \nabla \ue \right) + y^{\alpha} c\ue
$
where, for all $(x',y) \in \C $,  $\mathbf{A}(x',y) =  \textrm{diag} \{A(x'),1\} \in C^{0,1}(\C,\GL(n+1,\R))$. Consequently, we can rewrite problem \eqref{alpha_harm_L}
as follows:
\begin{equation}
\label{alpha_harm_Ly}
  -\DIV \left( y^{\alpha} \mathbf{A} \nabla \ue \right) + y^{\alpha} c\ue = 0  \textrm{ in } \C, \quad
  \ue = 0 \text{ on } \partial_L \C, \quad
  \frac{ \partial \ue }{\partial \nu^\alpha} = d_s \zsf  \text{ on } \Omega \times \{0\}. 
\end{equation}

Before proceeding with the description and analysis of our method, let us give an overview of those advocated in the literature. The study of solution techniques for problems involving fractional diffusion is a relatively new but rapidly growing area of research. We refer to \cite{NOS,NOS3} for an overview of the state of the art. 

Numerical strategies for solving a discrete optimal control problem with PDE constraints have been widely studied in the literature; see \cite{HPUU:09,HT:10,IK:08} for an extensive list of references. They are mainly divided in two categories, which rely on an agnostic discretization of the state and adjoint equations. They differ on whether or not the admissible control set is also discretized. The first approach \cite{HAntil_RHNochetto_PSodre_2014b,ACT:02,CT:05,R:06} discretizes the admissible control set. 
The second approach \cite{Hinze:05} induces a discretization of the optimal control by projecting the discrete adjoint state into the admissible control set. Mainly, these studies are concerned with control problems governed by elliptic and parabolic PDEs, both linear and semilinear. The common feature here is that, in contrast to
\eqref{fractional}, the state equation is local. To the best of our knowledge, this is the first work addressing the numerical approximation of an optimal control problem involving fractional powers of elliptic operators in general domains. 
\EO{For a comprehensive treatment of a fractional space-time optimal control problem we refer to our recently submitted paper \cite{AOS}.}

The main contribution of this work is the study of solution techniques for problem \eqref{Jintro}-\eqref{cc}. We overcome the nonlocality of the operator $\mathcal{L}^s$ by using the Caffarelli-Silvestre extension \cite{CS:07}. To be concrete, we consider the equivalent formulation: 
\begin{equation*}
 \text{min }J(\ue|_{y=0},\zsf)= \frac{1}{2}\| \ue|_{y=0} - \usf_{d} \|^2_{L^2(\Omega)} + \frac{\mu}{2} \| \zsf\|^2_{L^2(\Omega)},
\end{equation*}
subject to \eqref{alpha_harm_Ly} and \eqref{cc}. We will refer to the optimal control problem described above as the \emph{extended optimal control problem}; see \S\ref{sub:control_extended} for a precise definition.

Inspired by \cite{NOS}, we propose the following simple strategy to find the solution to the fractional optimal control problem \eqref{Jintro}-\eqref{cc}: given  $s\in (0,1)$, and a desired state $\usf_d: \Omega \rightarrow \R$, we solve the equivalent extended control problem, thus obtaining an optimal control $\bar{\zsf}(x')$ and an optimal state $\bar{\ue}: (x',y) \in \C \mapsto \bar{\ue}(x',y) \in \R$. Setting $\bar{\usf}: x' \in \Omega  \mapsto \bar{\usf}(x') = \bar{\ue}(x',0) \in \R$, we obtain the optimal pair $(\bar{\usf},\bar{\zsf})$ solving the fractional optimal control problem \eqref{Jintro}-\eqref{cc}. 

In this paper we propose and analyze two discrete schemes to solve \eqref{Jintro}-\eqref{cc}. Both of them rely on a discretization of the state equation \eqref{alpha_harm_Ly} and the corresponding adjoint equation via first degree tensor product finite elements on anisotropic meshes as in \cite{NOS}. However they differ on whether or not the set of controls is discretized as well. The first approach is semi-discrete and is based on the so-called variational approach \cite{Hinze:05}: the set of controls is not discretized. The second approach is fully discrete and discretizes the set of controls by piecewise constant functions \cite{ACT:02,CT:05,R:06}.

The outline of this paper is as follows. In \S\ref{sec:Prelim} we introduce some terminology used throughout this work. We recall the definition of the fractional powers of elliptic operators via spectral theory in \S\ref{sub:fractional_L}, and in \S\ref{sub:CaffarelliSilvestre} we introduce the functional framework that is suitable to analyze problems \eqref{fractional} and \eqref{alpha_harm_Ly}. In \S\ref{sec:control} we define the \emph{fractional} and \emph{extended optimal control problems}. For both of them, we derive existence and uniqueness results together with first order necessary and sufficient optimality conditions. We prove that both problems are equivalent. In addition, we study the regularity properties of the optimal control. The numerical analysis of the \emph{fractional control problem} begins in \S \ref{sec:control_truncated}.  Here we introduce a truncation of the state equation \eqref{alpha_harm_Ly}, and propose the \emph{truncated optimal control problem}. We derive approximation properties of its solution. Section \ref{sec:apriori} is devoted to the study of discretization techniques to solve the fractional control problem. In \S\ref{subsec:state_equation} we review the a priori error analysis developed in \cite{NOS} for the state equation \eqref{alpha_harm_Ly}. In \S\ref{subsec:va} we propose a semi-discrete scheme for the fractional control problem, and derive a priori error estimate for both the optimal control and state. In \S\ref{subsec:fd}, we propose a 
fully-discrete scheme for the control problem \eqref{Jintro}-\eqref{cc} and derive a priori error estimates for the optimal variables. Finally, in \S\ref{sec:numerics}, we present numerical experiments that illustrate the theory developed in \S\ref{subsec:fd} and reveal a competitive performance of anisotropic over quasi-uniform.

\section{Notation and preliminaries}
\label{sec:Prelim}

\subsection{Notation}
\label{sub:notation}

Throughout this work $\Omega$ is an open, bounded and connected domain of $\R^n$, $n\geq1$, with polyhedral boundary $\partial\Omega$. We define the semi-infinite cylinder with base $\Omega$ and its lateral boundary, respectively, by
$
\C = \Omega \times (0,\infty)
$
and 
$
\partial_L \C  = \partial \Omega \times [0,\infty).
$
Given $\Y>0$, we define the truncated cylinder 
$
  \C_\Y = \Omega \times (0,\Y)
$
and $\partial_L\C_\Y$ accordingly. 

Throughout our discussion we will be dealing with objects defined in $\R^{n+1}$, then it will be convenient to distinguish the extended dimension. A vector $x\in \R^{n+1}$, will be denoted by
$
  x =  (x^1,\ldots,x^n, x^{n+1}) = (x', x^{n+1}) = (x',y),
$
with $x^i \in \R$ for $i=1,\ldots,{n+1}$, $x' \in \R^n$ and $y\in\R$.

We denote by $\calL^s$, $s \in (0,1)$, a fractional power of the second order, symmetric and uniformly elliptic operator $\calL$. The parameter $\alpha$ belongs to $(-1,1)$ and is related to the power $s$ of the fractional operator $\calL^s$ by the formula $\alpha = 1 -2s$.

If $\Xcal$ and $\Ycal$ are normed vector spaces, we write $\Xcal \hookrightarrow \Ycal$ to denote that $\Xcal$ is continuously embedded in $\Ycal$. We denote by $\Xcal'$ the dual of $\Xcal$ and by $\|\cdot\|_{\Xcal}$ the norm of $\Xcal$. Finally, the relation $a \lesssim b$ indicates that $a \leq Cb$, with a constant $C$ that does not depend on $a$ or $b$ nor the discretization parameters. The value of $C$ might change at each occurrence. 

\subsection{Fractional powers of general second order elliptic operators}
\label{sub:fractional_L}
Our definition is based on spectral theory \cite{CT:10,CDDS:11}. The operator $\mathcal{L}^{-1}: L^2(\Omega)\to L^2(\Omega)$, which solves $\mathcal{L} w  = f$ in $\Omega$ and $w = 0$ on $\partial \Omega$, is compact, symmetric and positive, so its spectrum $\{\lambda_k^{-1} \}_{k\in \mathbb N}$ is discrete, real, positive and accumulates at zero. Moreover, the eigenfunctions:
\begin{equation}
  \label{eigenvalue_problem_L}
    \mathcal{L} \varphi_k = \lambda_k \varphi_k  \text{ in } \Omega,
    \qquad
    \varphi_k = 0 \text{ on } \partial\Omega, \qquad k \in \mathbb{N},
\end{equation}
form an orthonormal basis of $L^2(\Omega)$. Fractional powers of $\mathcal L$ can be defined by
\begin{equation}
  \label{def:second_frac}
  \mathcal{L}^s w  := \sum_{k=1}^\infty \lambda_k^{s} w_k \varphi_k, \qquad w \in C_0^{\infty}(\Omega), \qquad s \in (0,1),
\end{equation} 
where $w_k = \int_{\Omega} w \varphi_k $. By density, this definition can be extended to the space
\begin{equation}
\label{def:Hs}
  \Hs = \left\{ w = \sum_{k=1}^\infty w_k \varphi_k: 
  \sum_{k=1}^{\infty} \lambda_k^s w_k^2 < \infty \right\}
    =
  \begin{dcases}
    H^s(\Omega),  & s \in (0,\sr), \\
    H_{00}^{1/2}(\Omega), & s = \sr, \\
    H_0^s(\Omega), & s \in (\sr,1).
  \end{dcases}
\end{equation}
The characterization given by the second equality is shown in \cite[Chapter 1]{Lions}; see \cite{BSV} and \cite[\S~2]{NOS} for a discussion. The space $\mathbb{H}^{1/2}(\Omega)$ is the so-called \emph{Lions-Magenes} space, which can be characterized as (\cite[Theorem~11.7]{Lions} and \cite[Chapter 33]{Tartar})
\begin{equation*}
  \mathbb{H}^{1/2}(\Omega) = \left\{ w \in H^{\srn}(\Omega):
        \int_{\Omega} \frac{w^2(x')}{\textrm{dist}(x',\partial \Omega)} \diff x'  < \infty  \right\}.
\end{equation*}
For $ s \in (0,1)$ we denote by $\Hsd$ the dual space of $\Hs$.

\subsection{The Caffarelli-Silvestre extension problem}
\label{sub:CaffarelliSilvestre}
The Caffarelli-Silvestre result \cite{CS:07}, or its variants \cite{CT:10, CDDS:11}, requires to address the nonuniformly 
elliptic equation \eqref{alpha_harm_Ly}. To this end, we consider weighted Sobolev spaces with the weight $|y|^{\alpha}$, $\alpha \in (-1,1)$. If $D \subset \R^{n+1}$, we define $L^2(|y|^{\alpha},D)$ as the space of all measurable functions defined on $D$ such that 
$
\| w \|_{L^2(|y|^{\alpha},D)}^2 = \int_{D}|y|^{\alpha} w^2 < \infty,
$
and
\[
H^1(|y|^{\alpha},D) =
  \left\{ w \in L^2(|y|^{\alpha},D): | \nabla w | \in L^2(|y|^{\alpha},D) \right\},
\]
where $\nabla w$ is the distributional gradient of $w$. We equip $H^1(|y|^{\alpha},D)$ with the norm
\begin{equation}
\label{wH1norm}
\| w \|_{H^1(|y|^{\alpha},D)} =
\Big(  \| w \|^2_{L^2(|y|^{\alpha},D)} + \| \nabla w \|^2_{L^2(|y|^{\alpha},D)} \Big)^{\frac{1}{2}}.
\end{equation}

Since $\alpha \in (-1,1)$ we have that $|y|^\alpha$ belongs to the so-called Muckenhoupt class $A_2(\R^{n+1})$; see \cite{GU,Turesson}. This, in particular, implies that $H^1(|y|^{\alpha},D)$ equipped with the norm \eqref{wH1norm}, is a Hilbert space and the set $C^{\infty}(D) \cap H^1(|y|^{\alpha},D)$ is dense in $H^1(|y|^{\alpha},D)$ (cf.~\cite[Proposition 2.1.2, Corollary 2.1.6]{Turesson}, \cite{KO84} and \cite[Theorem~1]{GU}). We recall now the definition of Muckenhoupt classes; see \cite{GU,Turesson}.

\begin{definition}[Muckenhoupt class $A_2$]
 \label{def:Muckenhoupt}
Let $\omega$ be a weight and $N \geq 1$. We say $\omega \in A_2(\R^N)$ if
\begin{equation*}
  C_{2,\omega} = \sup_{B} \left( \fint_{B} \omega \right)
            \left( \fint_{B} \omega^{-1} \right) < \infty,
\end{equation*}
where the supremum is taken over all balls $B$ in $\R^N$.
\end{definition}

To study the extended control problem, we define the weighted Sobolev space
\begin{equation}
  \label{HL10}
  \HL(y^{\alpha},\C) = \left\{ w \in H^1(y^\alpha,\C): w = 0 \textrm{ on } \partial_L \C\right\}.
\end{equation}
As \cite[(2.21)]{NOS} shows, the following \emph{weighted Poincar\'e inequality} holds:
\begin{equation}
\label{Poincare_ineq}
\| w \|_{L^2(y^{\alpha},\C)} \lesssim \| \nabla w \|_{L^2(y^{\alpha},\C)},
\quad \forall w \in \HL(y^{\alpha},\C).
\end{equation}
Then $\| \nabla w \|_{L^2(y^{\alpha},\C)}$ is equivalent to \eqref{wH1norm} in $\HL(y^{\alpha},\C)$. For $w \in H^1( y^{\alpha},\C)$, we denote by $\tr w$ its trace onto $\Omega \times \{ 0 \}$, and we recall (\cite[Prop.~2.5]{NOS})
\begin{equation}
\label{Trace_estimate}
\tr \HL(y^\alpha,\C) = \Hs,
\qquad
  \|\tr w\|_{\Hs} \leq C_{\tr} \| w \|_{\HLn(y^\alpha,\C)}.
\end{equation}

Let us now describe the Caffarelli-Silvestre result and its extension to second order operators; \cite{CS:07,ST:10}.  Consider a function $u$ defined on $\Omega$. We define the $\alpha$-harmonic extension of $u$ to the cylinder $\C$, as the function $\ue$ that solves the problem
\begin{equation}
\label{alpha_harm_Lyu}
\begin{dcases}
  -\DIV \left( y^{\alpha} \mathbf{A} \nabla \ue \right) + y^{\alpha} c\ue = 0 & \textrm{in } \C,\\
  \ue = 0 \quad  \text{on } \partial_L \C, &
  \ue = u \quad \text{on } \Omega \times \{0\}. \\
\end{dcases}
\end{equation}
Problem \eqref{alpha_harm_Lyu} has a unique solution $\ue \in \HL(y^\alpha,\C)$ whenever $u\in \Hs$. We define the \emph{Dirichlet-to-Neumann} operator $\textrm{N} : \Hs \to \Hsd$ 
\[
   u \in \Hs \longmapsto
   \textrm{N}(u) = \frac{\partial \ue}{\partial \nu^{\alpha}} \in \Hsd,
\]
where $\ue$ solves \eqref{alpha_harm_Lyu} and $\tfrac{\partial \ue}{\partial \nu^{\alpha}}$ is given in \eqref{def:lf}.
The fundamental result of \cite{CS:07}, see also \cite[Lemma~2.2]{CDDS:11} and \cite[Theorem 1.1]{ST:10}, is stated below.

\begin{theorem}[Caffarelli--Silvestre extension]
\label{TH:CS}
If $s\in(0,1)$ and $u \in \Hs$, then
\[
  d_s \Laps u = \mathrm{N}(u),
\]
in the sense of distributions. Here $\alpha = 1-2s$ and
$
 d_s = 2^{\alpha} \frac{\Gamma(1-s)}{\Gamma(s)},
$
where $\Gamma$ denotes the Gamma function.
\end{theorem}

The relation between the fractional Laplacian and the extension problem is now clear. Given $\zsf \in \Hsd$, a function $u \in \Hs$ solves \eqref{fractional} if and only if its $\alpha$-harmonic extension $\ue \in \HL(y^{\alpha},\C)$ solves \eqref{alpha_harm_Ly}.

We now present the weak formulation of \eqref{alpha_harm_Ly}: Find $\ue \in \HL(y^{\alpha},\C)$ such that
\begin{equation}
\label{alpha_harm_L_weak}
  a(\ue ,\phi) 
  =  \langle \zsf , \tr \phi \rangle_{\Hsd \times \Hs} \quad \forall \phi \in \HL(y^{\alpha},\C),
\end{equation}
where $\HL(y^{\alpha},\C)$ is as in \eqref{HL10}. For $w, \phi \in \HL(y^{\alpha},\C)$, the bilinear form $a$ is defined by
\begin{equation}
\label{a}
a(w,\phi) =  \frac{1}{d_s} \int_{\C} {y^{\alpha}\mathbf{A}(x',y)} \nabla w \cdot \nabla \phi
 + y^{\alpha} c(x') w \phi,
\end{equation}
where $\langle \cdot, \cdot \rangle_{\Hsd \times \Hs}$ denotes the duality pairing between $\Hs$ and $\Hsd$ which, as a consequence of \eqref{Trace_estimate}, is well defined for $\zsf \in \Hsd$ and $\phi \in \HL(y^{\alpha},\C)$.
 
\begin{remark}[equivalent semi-norm]
\label{rem:equivalent} \rm
Notice that the regularity assumed for $A$ and $c$, together with the weighted  Poincar\'e inequality \eqref{Poincare_ineq}, imply that the bilinear form $a$, defined in \eqref{a}, is bounded and coercive in $\HL(y^\alpha \C)$. In what follows we shall use repeatedly the fact that $a(w,w)^{1/2}$ is a norm equivalent to $|\cdot|_{\HLn(y^\alpha,\C)}$.
\end{remark}

Remark \eqref{rem:equivalent} in conjunction with \cite[Proposition 2.1]{CDDS:11} for $s \in (0,1)\setminus\{\srn\}$ and \cite[Proposition 2.1]{CT:10} for $s = \srn $ provide us the following estimates for problem \eqref{alpha_harm_L_weak}:
\begin{equation}
\label{estimate_s}
  \| \ue \|_{\HLn(\C,y^{\alpha})} \lesssim \| \usf \|_{\Hs} \lesssim \| \zsf \|_{\Hsd}.
\end{equation}

We conclude with a representation of the solution of problem \eqref{alpha_harm_L_weak} using the eigenpairs $\{ \lambda_k, \varphi_k\}$ defined in \eqref{eigenvalue_problem_L}. Let the solution to \eqref{fractional} be given by $\usf(x')=\sum_k \usf_k\varphi_k(x')$. The solution $\ue$ of problem \eqref{alpha_harm_L_weak} can then be written as
\begin{equation*}
  \ue(x,t) = \sum_{k=1}^\infty \usf_k \varphi_k(x') \psi_k(y),
\end{equation*}
where $\psi_k$ solves
\begin{equation}
\label{psik}
\psi_k'' + \alpha y^{-1}\psi_k' - \lambda_k \psi_k = 0, \quad \psi_k(0) = 1, \quad \psi_k(y) \to 0 \mbox{ as } y \to \infty.
\end{equation}
If $s=\sr$, then clearly $\psi_k(y) = e^{-\sqrt{\lambda_k}y}$. For $s \in (0,1)\setminus\{\srn\}$ we have that if $c_s = \tfrac{2^{1-s}}{\Gamma(s)}$, then
$
  \psi_k(y) = c_s \left(\sqrt{\lambda_k}y\right)^s K_s(\sqrt{\lambda_k} y)
$
(\cite[Proposition~2.1]{CDDS:11}), where $K_s$ is the modified Bessel function of the second kind \cite[Chapter~9.6]{Abra}.

\section{The optimal fractional and extended control problems}
\label{sec:control}

In this section we describe and analyze the \emph{fractional} and \emph{extended optimal control problems}. For both of them, we derive existence and uniqueness results together with first order necessary and sufficient optimality conditions. We conclude the section by stating the equivalence between both optimal control problems, which set the stage to propose and study numerical algorithms to solve the fractional optimal control problem.

\subsection{The optimal fractional control problem}
\label{sub:control_fractional}
We start by recalling the \emph{fractional control problem} introduced in \S\ref{sec:introduccion}, which, given the functional $J$ defined in \eqref{functional}, reads as follows: Find
$
\text{min } J(\usf,\zsf),
$
subject to the fractional state equation \eqref{fractional} and the control constraints \eqref{cc}. The set of \textit{admissible controls} $\Zad$ is defined by
\begin{equation}
 \label{ac}
 \Zad= \{ \wsf \in L^2(\Omega): \asf(x') \leq \wsf(x') \leq \bsf(x'), \textrm{~~~a.e~~~}  x' \in \Omega \},
\end{equation}
where $\asf,\bsf \in L^2(\Omega)$ and satisfy $\asf(x') \leq \bsf(x')$ a.e. $x' \in \Omega$. The function $\usf_d \in L^2(\Omega)$ denotes the desired state and $\mu > 0$ the so-called regularization parameter.

In order to study the existence and uniqueness of this problem, we follow \cite[\S~2.5]{Tbook} and introduce the so-called fractional control-to-state operator. 

\begin{definition}[fractional control-to-state operator]
\label{def:fractional_operator}
We define the fractional control to state operator $\mathbf{S}: \Hsd \rightarrow \Hs$ such that
for a given control $\zsf \in \Hsd$ it associates a unique state $\usf(\zsf) \in \Hs$ via
the state equation \eqref{fractional}.
\end{definition}

As a consequence of \eqref{estimate_s}, $\mathbf{S}$ is a linear and continuous mapping from $\Hsd$ into $\Hs$. Moreover, in view of the continuous embedding $\Hs \hookrightarrow L^2(\Omega) \hookrightarrow \Hsd$, we may also consider $\mathbf{S}$ acting from $L^2(\Omega)$ and with range in $L^2(\Omega)$. For simplicity, we keep the notation $\mathbf{S}$ for such an operator.

We define the fractional optimal state-control pair as follows.

\begin{definition}[fractional optimal state-control pair]
\label{def:op_pair}
A state-control pair $(\ousf(\ozsf),\ozsf)$ $\in \Hs \times \Zad$ is called optimal for problem
\eqref{Jintro}-\eqref{cc}, if $\ousf(\ozsf) = \mathbf{S} \ozsf$ and
\[
 J(\ousf(\ozsf),\ozsf) \leq J(\usf(\zsf),\zsf),
\]
for all $(\usf(\zsf),\zsf) \in \Hs \times \Zad$ such that $\usf(\zsf)= \mathbf{S} \zsf$.
\label{def:fractional_pair}
\end{definition}

Given that $\Lcal^s$ is a self-adjoint operator, it follows that $\mathbf{S}$ is a self-adjoint operator as well. Consequently, we have the following definition for the adjoint state.

\begin{definition}[fractional adjoint state]
Given a control $\zsf \in \Hsd$, we define the fractional adjoint state $\psf(\zsf) \in \Hs$ as $\psf(\zsf) = \mathbf{S}(\usf - \usf_d)$.
\label{def:fractional_adjoint} 
\end{definition}

We now present the following result about existence and uniqueness of the optimal control together with the first order necessary and sufficient optimality conditions.

\begin{theorem}[existence, uniqueness and optimality conditions]
The fractional optimal control problem \eqref{Jintro}-\eqref{cc} has a unique optimal solution $(\ousf, \ozsf)  $ 
$\in$ $ \Hs \times \Zad$. The optimality conditions 
$\ousf = \mathbf{S} \bar{\zsf} \in \Hs$, $\opsf = \mathbf{S} ( \ousf - \usf_{\textrm{d} } ) \in \Hs$, and
\begin{equation}
 \label{vi_fractional}
 \ozsf \in \Zad, \qquad(\mu \ozsf + \opsf, \zsf - \ozsf )_{L^2(\Omega)} \geq 0 \quad \forall \zsf \in \Zad
\end{equation}
hold. These conditions are necessary and sufficient. 
\label{TH:fractional_control}
\end{theorem}
\begin{proof}
We start by noticing that using the control-to-state operator $\mathbf{S}$, the fractional control problem reduces to the following quadratic optimization problem:
 \[
  \min_{\zsf \in \Zad} f(\zsf): = \frac{1}{2} \| \mathbf{S} \zsf - \usf_{\textrm{d}}\|^2_{L^2(\Omega)}
  + \frac{\mu}{2} \|\zsf \|^2_{L^2(\Omega)}.
 \]
Since $\mu > 0$, it is immediate that the functional $f$ is strictly convex. Moreover, the set $\Zad$ is nonempty, closed, bounded and convex in $L^2(\Omega)$. Then, 
invoking an infimizing sequence argument, followed by the well-posedness of the state equation, we derive the existence of an optimal control $\ozsf \in \Zad$; see \cite[Theorem 2.14]{Tbook}. The uniqueness of $\ozsf$ is a consequence of the strict convexity of $f$.
The first order optimality condition \eqref{vi_fractional} is a direct consequence of \cite[Theorem 2.22]{Tbook}.
\end{proof}

In what follows, we will, without explicit mention, make the following regularity assumption concerning the domain $\Omega$:
\begin{equation}
\label{Omega_regular}
 \| w \|_{H^2(\Omega)} \lesssim \| \mathcal{L} w \|_{L^2(\Omega)}, \quad \forall w \in H^2(\Omega) \cap H^1_0(\Omega), 
\end{equation}
which is valid, for instance, if the domain $\Omega$ is convex \cite{Grisvard}. \EO{In addition, for some values of $s \in (0,1)$, we will need the following assumption on $\asf$ and $\bsf$ defining $\Zad$:}
\begin{equation}
\label{ab_condition}
\EO{\asf \leq 0 \leq \bsf \textrm{ on } \partial \Omega.}
\end{equation}
\EO{The range of values of $s$ for which such a condition is needed it will be explicitly stated in the subsequent results.}

We conclude with the study of the regularity properties of the \emph{fractional optimal control} $\ozsf$. These properties are fundamental to derive a priori error estimates for the discrete algorithms proposed in \S\ref{subsec:va} and \S\ref{subsec:fd}. To do that, we recall the following result: If $\mu >0$ and $\opsf$ is given by Definition \ref{def:fractional_adjoint}, then the projection formula
\begin{equation}
\label{projection_formula}
    \ozsf(x') = \textrm{proj}_{[\asf(x'),\bsf(x')]} \left(- \frac{1}{\mu} \opsf(x') \right)
\end{equation}
is equivalent to the variational inequality \eqref{vi_fractional}; see \cite[Section 2.8]{Tbook} for details. In the formula  previously defined $\textrm{proj}_{[\asf,\bsf]}(v) = \min\{ \bsf, \max \{ \asf,v \} \}$.
\begin{lemma}[$H^1$-regularity of control]
\label{LM:reg_control}
Let $\ozsf \in \Zad$ be the fractional optimal control, $\asf,\bsf \in H^1(\Omega)$ and $\usf_d \in \Ws$. If $s \in [\frac{1}{4},1)$, then $\ozsf \in H^{1}(\Omega)$. If $s \in (0,\frac{1}{4})$ and, in addition, \EO{\eqref{ab_condition} holds}, then $\EO{\ozsf} \in H^1_0(\Omega)$.
\end{lemma}
\begin{proof}
Let $\phi \in \Hs$ be the solution to $\mathcal{L}^s \phi = \ozsf$ in $\Omega$ and $\phi = 0$ on $\partial \Omega$. Since $\Omega$ satisfies \eqref{Omega_regular}, $\mathcal{L}^s$ is is a pseudodifferential operator of order $2s$ and $\ozsf \in L^2(\Omega)$, we conclude that $\phi \in \mathbb{H}^{2s}(\Omega)$. Consequently, if $\ousf = \ousf(\ozsf)$ solves \eqref{fractional} and $\opsf = \opsf(\ozsf)$ is given by the Definition \ref{def:fractional_adjoint}, then
\[
\EO{\ousf (\ozsf) \in \mathbb{H}^{2s}(\Omega), \qquad \opsf(\ozsf) \in \mathbb{H}^{\kappa}(\Omega),}
\]
where, since $\usf_d \in \mathbb{H}^{1-s}(\Omega)$, $\kappa = \min\{4s,1+s \} < 2$. Next, we define  
$A \wsf = \max \{ \wsf,0 \}$, which satisfies:
\begin{enumerate}[(a)]
 \item \label{op:a} $\| A \wsf \|_{H^1(\Omega)} \lesssim \| \wsf \|_{H^1(\Omega)}$ for all $\wsf \in H^1(\Omega)$
 \cite[Theorem A.1]{KSbook}.
 \item \label{op:b} $\| A\wsf_1 - A\wsf_2\|_{L^2(\Omega)} \leq \| \wsf_1 - \wsf_2\|_{L^2(\Omega)}$ for all $\wsf_1, \wsf_2 \in L^2(\Omega)$.
\end{enumerate}

\noindent \boxed{1}
$s \in [\frac{1}{4},1)$ and $\asf, \bsf \in H^1(\Omega)$: The formula \eqref{projection_formula}, in conjunction with property \eqref{op:a} immediately implies that $\ozsf \in H^1(\Omega)$; see also \cite[Theorem 2.37]{Tbook}.

\noindent \boxed{2}
$s \in (0,\frac{1}{4})$, $\asf, \bsf \in H^1(\Omega)$ and \eqref{ab_condition} holds: In this case $\bar{\psf}(\ozsf) \in \mathbb{H}^{4s}(\Omega)$. As the operator $A$ satisfies \eqref{op:a} and \eqref{op:b}, an interpolation argument based on \cite[Lemma 28.1]{Tartar} allows us to conclude $\| A\bar{\psf} \|_{\mathbb{H}^{4s}(\Omega)} \lesssim \| \opsf \|_{\mathbb{H}^{4s}(\Omega)}$. This, in view of \eqref{projection_formula},
and the fact that  $\asf, \bsf \in H^1(\Omega)$ satisfy \eqref{ab_condition}, immediately implies that $\ozsf \in \mathbb{H}^{4s}(\Omega)$. We now consider two cases:

\noindent \boxed{2.1}
\EO{$s \in [\frac{1}{8},\frac{1}{4}):$} Since $\ozsf \in \mathbb{H}^{4s}(\Omega)$, we conclude that $\ousf (\ozsf) \in \mathbb{H}^{6s}(\Omega)$ and $\opsf(\ozsf) \in \mathbb{H}^{\sigma}(\Omega)$ with $\sigma = \min\{\EO{8s},1+s\}$. Then, in view of \eqref{projection_formula}, we have that $\ozsf \in H_0^1(\Omega)$ for $s \in [\frac{1}{8},\frac{1}{4})$. 

\noindent \boxed{2.2} \EO{$s \in (0,\frac{1}{8}):$ A nonlinear operator interpolation argument, again, yields $\ozsf \in \mathbb{H}^{8s}(\Omega)$. Consequently,
$\ousf(\ozsf) \in \mathbb{H}^{10s}(\Omega)$ and $\opsf(\ozsf) \in \mathbb{H}^{\delta}(\Omega)$ where $\delta = \min \{ 12s,1+s\}$.}

\noindent \boxed{2.2.1} \EO{$s \in [\frac{1}{12},\frac{1}{8}):$ We immediately conclude that $\ozsf \in H_0^1(\Omega)$.}

\noindent \boxed{2.2.2} \EO{$s \in (0,\frac{1}{12}):$ Proceed as before.}

\EO{Proceeding in this way we can conclude, after a finite number of steps, that for any $s \in (0,\frac{1}{4})$ we have $\ozsf \in H_0^1(\Omega)$. This concludes the proof.}
\end{proof}

\begin{remark}[regularity of the fractional optimal state and adjoint]
\label{rk:reg_states}
\rm
\EO{Let $\asf,\bsf \in H^1(\Omega)$ and $\usf_d \in \Ws$. Then, as a consequence of the analysis developed in Lemma \ref{LM:reg_control}, we conclude the following regularity results for $\ousf = \ousf(\ozsf)$ and $\opsf = \opsf(\ozsf)$. If $s \in [\frac{1}{2},1)$, then $\ousf \in H_0^1(\Omega)$. If $s \in (0,\frac{1}{2})$ and \eqref{ab_condition} holds, then $\ousf \in H_0^1(\Omega)$. On the other hand, if $s \in [\frac{1}{4},1)$, then $\opsf \in H^1_0(\Omega)$ and, if $s \in (0,\frac{1}{4})$ and \eqref{ab_condition} holds, $\opsf(\ozsf) \in H^1_0(\Omega)$.}
\end{remark}

\subsection{The optimal extended control problem}
\label{sub:control_extended}

In order to overcome the nonlocality feature in the fractional control problem we introduce an equivalent problem: the \emph{extended optimal control problem}. The main advantage of the latter, which was already motivated in \S\ref{sec:introduccion}, is its local nature. We define the \emph{extended optimal control problem} as follows: Find
$
 \text{min } \EO{J(\tr \ue,\qsf)},
$
subject to the state equation
\begin{equation}
\label{extension_weak}
  a(\ue,\phi) = \langle \qsf, \tr \phi \rangle_{\Hsd \times \Hs}
    \quad \forall \phi \in \HL(y^{\alpha},\C)
\end{equation}
and the control constraints 
$
 \qsf \in \Zad,
$
where the functional $J$ is defined by \eqref{Jintro}.
\begin{definition}[extended control-to-state operator]
\label{def:extended_operator}
The map $\mathbf{G}: \Hsd \ni q \mapsto \tr \ue(\qsf) \in \Hs $ where $ \ue(\qsf) \in \HL(y^{\alpha},\C)$ solves \eqref{extension_weak} is called the extended control to state operator.
\end{definition}
\begin{remark}[The operators $\mathbf{S}$ and $\mathbf{G}$ coincide]\rm
\label{rk:SandG}
The result of Theorem \ref{TH:CS} tells us that if $\usf (\qsf) \in \Hs$ denotes the solution to \eqref{fractional} with $\qsf \in \Hsd$ as a datum, and $\ue(\qsf)$ solves \eqref{extension_weak}, then
\[
 \tr \ue (\qsf) = \usf(\qsf).
\]
Consequently, the action of the operators $\mathbf{S}$ and $\mathbf{G}$ coincide, and then the results of \S\ref{sub:control_fractional} imply that $\mathbf{G}$ is a well defined, linear and continuous operator. 
\end{remark}
\begin{definition}[extended optimal state-control pair]
\label{def:op_pair_e}
A state-control pair $(\oue(\oqsf),\oqsf)$ $\in \HL(y^{\alpha},\C) \times \Zad$ is called optimal for 
the extended optimal control problem if $\tr \oue(\oqsf) = \mathbf{G} \oqsf$ and
\[
 \EO{J(\tr \oue(\oqsf),\oqsf) \leq J(\tr \ue(\qsf),\qsf),}
\]
for all $(\ue(\qsf),\qsf) \in \HL(y^{\alpha},\C)  \times \Zad$ such that $\tr \ue(\qsf) = \mathbf{G} \qsf$. 
\end{definition}

\EO{Since $\mathbf{G}$ is self-adjoint with respect to the standard $L^2(\Omega)$-inner product, we have the following definition for the extended optimal adjoint state.}

\begin{definition}[extended optimal adjoint state]
\label{def:extension_adjoint}
The extended optimal adjoint state $\ope(\oqsf) \in \HL(y^{\alpha},\C)$, associated with  
$\oue(\oqsf) \in \HL(y^{\alpha},\C)$, is the unique solution to
\begin{equation}
\label{extension_adjoint}
  a(\ope(\oqsf) ,\phi) = 
( \tr \oue(\oqsf)  - \usfd, \tr \phi )_{L^2(\Omega)},    
\end{equation}
for all $\phi \in \HL(y^{\alpha},\C)$.
\end{definition}

\EO{Definitions \ref{def:extended_operator} and \ref{def:extension_adjoint} yield: $\tr \ope(\oqsf)  =  \mathbf{G}( \mathbf{G} \EO{\bar{\qsf}} - \usf_d)$.} The existence and uniqueness of the extended optimal control problem, together with the first order necessary and sufficient optimality condition follow the arguments developed in Theorem \ref{TH:fractional_control}.

\begin{theorem}[existence, uniqueness and optimality system]
\label{TH:extended_control}
The extended optimal control problem has a unique optimal solution $(\oue, \oqsf)$ $\in$ $ \HL(y^{\alpha},\C) \times \Zad$. The optimality system
\begin{equation}
\begin{dcases}
 \oue = \oue(\oqsf) \in \HL(y^{\alpha},\C) \textrm{ solution of } \eqref{extension_weak}, \\
 \ope = \ope(\oqsf)  \in \HL(y^{\alpha},\C) \textrm{ solution of } \eqref{extension_adjoint}, \\
 \oqsf \in \Zad, \quad (\tr \bar{\pe} + \mu \oqsf , \qsf - \oqsf )_{L^2(\Omega)} \geq 0 \quad \forall \qsf \in \Zad,
\end{dcases}
\label{op_extended} 
\end{equation}
hold. These conditions are necessary and sufficient. 
\end{theorem}

We conclude this section with the equivalence between the optimal fractional and extended control problems.

\begin{theorem}[equivalence of the fractional and extended control problems]
\label{TM:equivalence}
If $(\ousf(\ozsf),\ozsf) \in \Hs \times \Zad$ and $(\oue(\oqsf),\oqsf) \in \HL(y^{\alpha},\C) \times \Zad$ denote the optimal solutions to the fractional and extended optimal control problems respectively, then
\[
 \ozsf = \oqsf \quad \textrm{ and } \quad \tr \oue = \ousf,
\]
that is, the two problems are equivalent.
\label{lemma:equivalence}
\end{theorem}
\begin{proof}
The proof is a direct consequence of Theorem \ref{TH:CS}; see Remark \ref{rk:SandG}.
\end{proof}

\section{A truncated optimal control problem}
\label{sec:control_truncated}

A first step towards the discretization is to truncate the domain $\C$. Since the optimal state $\bar{\ue}$ decays exponentially in the extended dimension $y$, we truncate $\C$ to $\C_{\Y} = \Omega \times (0,\Y)$, for a suitable truncation parameter $\Y$ and seek solutions in this bounded domain; see \cite[\S 3]{NOS}. The exponential decay of the optimal state $\oue(\ozsf)$ is the content of the next result.

\begin{proposition}[exponential decay]
\label{PR:energyYinf}
For every $\Y \geq 1$, the optimal state $\oue = \oue(\ozsf) \in \HL(y^{\alpha},\C)$,
solution to problem \eqref{extension_weak}, satisfies
\begin{equation}
\label{energyYinf}
  \|\nabla \oue \|_{L^2(y^{\alpha},\Omega \times (\Y,\infty))} \lesssim e^{-\sqrt{\lambda_1} \Y/2}
  \| \ozsf \|_{\Hsd},
\end{equation}
where $\lambda_1$ denotes the first eigenvalue of the operator $\mathcal{L}$.
\end{proposition}
\begin{proof}
The estimate \eqref{energyYinf} follows directly from  \cite[Proposition 3.1]{NOS} in conjunction with Remark \ref{rem:equivalent}.
\end{proof}

As a consequence of Proposition \ref{PR:energyYinf}, given a control $\rsf \in \Hsd$, we consider the following truncated state equation
\begin{equation}
\label{alpha_extension_truncated}
\begin{dcases}
 -\DIV \left( y^{\alpha} \mathbf{A} \nabla v\right) + y^{\alpha} cv = 0 & \textrm{in }  \C_{\Y}, \\
  v = 0 \quad \textrm{on} \quad \partial_L \C_{\Y} \cup \Omega \times \{ \Y \},
  & \frac{ \partial v }{\partial \nu^\alpha}  = d_s \rsf  \textrm{ on } \Omega \times \{ 0\},\\
\end{dcases}
\end{equation}
for $\Y$ sufficiently large. In order to write a weak formulation of \eqref{alpha_extension_truncated} and formulate an appropriate optimal control problem, we define the weighted Sobolev space 
\[
  \HL(y^{\alpha},\C_\Y) = \left\{ w \in H^1(y^\alpha,\C_\Y): w = 0 \text{ on }
    \partial_L \C_\Y \cup \Omega \times \{ \Y\} \right\},
\]
and for all $w,\phi \in \HL(y^{\alpha},\C_\Y)$, the bilinear form
\begin{equation*}
a_\Y(w,\phi) = \frac{1}{d_s} \int_{\C_\Y} {y^{\alpha}\mathbf{A}(x)} \nabla w \cdot \nabla \phi + y^{\alpha} c(x') w \phi.
\end{equation*}

We are now in position to define the \emph{truncated optimal control problem} as follows: Find
$
 \text{min } \EO{J(\tr v,\rsf)}
$
subject to the truncated state equation
\begin{equation}
\label{extension_weak_Y}
  a_\Y(v,\phi) = \langle \rsf, \tr \phi \rangle_{\Hsd \times \Hs}
    \quad \forall \phi \in \HL(y^{\alpha},\C_\Y)
\end{equation}
and the control constraints 
$
 \rsf \in \Zad.
$

Before analyzing the truncated control problem, we present the following result.

\begin{proposition}[exponential convergence]
\label{pr:exp_convergence}
If  $\ue(\rsf) \in \HL(y^{\alpha},\C)$ solves \eqref{extension_weak} 
with $\qsf = \rsf \in \Hsd$ and $v(\rsf) \in \HL(y^{\alpha},\C_{\Y})$ solves \eqref{extension_weak_Y}
then, for any $\Y \geq 1$, we have
\begin{align}
\label{v-v^Y}
  \| \nabla\left( \ue (\rsf) - v (\rsf) \right) \|_{L^2(y^{\alpha},\C)} & \lesssim e^{-\sqrt{\lambda_1} \Y/4} \| \rsf\|_{\Hsd},\\
\label{trv-v^Y}
  \| \tr \left( \ue(\rsf) - v(\rsf) \right) \|_{\Hs} & \lesssim e^{-\sqrt{\lambda_1} \Y/4} \| \rsf \|_{\Hsd},
\end{align}
where $\lambda_1$ denotes the first eigenvalue of the operator $\mathcal{L}$.
\end{proposition}
\begin{proof}
The estimate \eqref{v-v^Y} follows from \cite[Theorem 3.5]{NOS} and Remark \ref{rem:equivalent}.
On the other hand, the trace estimate \eqref{Trace_estimate}, in conjunction 
with \eqref{v-v^Y}, yields \eqref{trv-v^Y}.
\end{proof}

Now, we introduce the truncated control-to-state operator as follows.

\begin{definition}[truncated control-to-state operator]
\label{def:truncated_operator}
We define the truncated control-to-state operator $\mathbf{H}: \Hsd \rightarrow \Hs$ such that for a given control $\rsf \in \Hsd$ it associates a unique state $\tr v(\rsf) \in \Hs$ via
\eqref{extension_weak_Y}. 
\end{definition}

The operator above is well defined, linear and continuous; see \cite[Lemma 2.6]{CT:10} for $s=1/2$ and \cite[Proposition 2.1]{CDDS:11} for any $s \in (0,1) \setminus \{\tfrac{1}{2}\}$. 

\EO{The optimal truncated state-control pair is defined along the same lines of Definitions \ref{def:op_pair} and \ref{def:op_pair_e}.} We now define the truncated optimal adjoint state as follows.

\begin{definition}[truncated optimal adjoint state]
\label{def:truncated_adjoint0}
The optimal adjoint state $\bar{p}(\orsf) \in \HL(y^{\alpha},\C_{\Y})$, associated  with 
$\bar{v}(\orsf) \in \HL(y^{\alpha},\C_{\Y})$, is defined to be the solution to
\begin{equation}
\label{truncated_adjoint}
a_\Y(\bar{p}(\orsf),\phi) = (  \tr \bar{v}(\orsf) - \usfd, \tr \phi )_{L^2(\Omega)},  
\end{equation}
for all $\phi \in \HL(y^{\alpha},\C_{\Y})$.
\end{definition}

As a consequence of Definitions \ref{def:truncated_operator}, and \ref{def:truncated_adjoint0}, we have that \EO{$\tr \bar{p} = \mathbf{H}(\mathbf{H} \orsf - \usfd)$.} In addition, the arguments developed in \S\ref{sub:control_fractional} allow us to conclude the following result.

\begin{theorem}[existence, uniqueness and optimality system]
\label{TH:truncated_control}
The truncated optimal control problem has a unique solution $(\bar{v}, \orsf)  $ $\in$ $ \HL(y^{\alpha},\C_{\Y}) \times \Zad$. The optimality system
\begin{equation}
\begin{dcases}
 \bar{v} = \bar{v}(\orsf) \in \HL(y^{\alpha},\C_{\Y}) \textrm{ solution of } \eqref{extension_weak_Y}, \\
 \bar{p} = \bar{p}(\orsf) \in \HL(y^{\alpha},\C_{\Y}) \textrm{ solution of } \eqref{truncated_adjoint}, \\
 \orsf \in \Zad, \quad (\tr \bar{p} + \mu \orsf , \rsf - \orsf )_{L^2(\Omega)} \geq 0 \quad \forall \rsf \in \Zad,
\end{dcases}
\label{op_truncated} 
\end{equation}
hold. These conditions are necessary and sufficient. 
\end{theorem}

The next result shows the approximation properties of the optimal pair $(\bar{v}(\orsf),\orsf)$ solving the truncated control problem.

\begin{lemma}[exponential convergence]
\label{LE:exp_convergence}
For every $\Y \geq 1$, we have
\begin{equation}
 \| \orsf - \ozsf \|_{L^2(\Omega)} \lesssim  e^{-\sqrt{\lambda_1} \Y/4} \EO{ \| \orsf \|_{L^2(\Omega)}},
\label{control_exp}
\end{equation}
and
\begin{equation}
\| \tr( \bar{\ue} - \bar{v}) \|_{L^2(\Omega)} 
\lesssim  e^{-\sqrt{\lambda_1} \Y/4}\EO{ \| \orsf \|_{L^2(\Omega)}}
\label{state_exp}
\end{equation}
\end{lemma}

\begin{proof} We proceed in four steps.

\noindent \boxed{1} We start by setting $\qsf = \orsf \in \Zad$ and $\rsf = \ozsf \in \Zad$ in the variational inequalities of the systems \eqref{op_extended} and \eqref{op_truncated} respectively. Adding the obtained results, we derive
\[
\mu \| \orsf - \ozsf \|^2_{L^2(\Omega)} 
\leq ( \tr (\bar{\pe} -  \bar{p}), \orsf - \ozsf )_{L^2(\Omega)},
\]
where we used Theorem \ref{TM:equivalence}: $\oqsf = \ozsf$. Now, we add and subtract $\tr \pe(\orsf)$ to arrive at
\[
 \mu \| \orsf - \ozsf \|^2_{L^2(\Omega)} 
\leq \left(  \tr (\bar{\pe} -  \pe(\orsf) ), \orsf - \ozsf \right)_{L^2(\Omega)}
+   \left( \tr ( \pe(\orsf) -  \bar{p}), \orsf - \ozsf \right)_{L^2(\Omega)} = \textrm{I} + \textrm{II}.
\]

\noindent \boxed{2} We estimate the term $\textrm{I}$ as follows: we use the relations $ \tr \ope(\ozsf)  =  \mathbf{G}( \mathbf{G} \ozsf - \usf_d)$ 
and $\tr \pe(\orsf)  = \mathbf{G}( \mathbf{G} \orsf - \usf_d)$ to obtain
\[
\textrm{I} = \left( \mathbf{G}(\mathbf{G} \ozsf - \mathbf{G}\orsf), \orsf - \ozsf\right)_{L^2(\Omega)}  
           = - \| \mathbf{G}( \ozsf - \orsf )\|_{L^2(\Omega)}^2 \leq 0.
\]
\noindent \boxed{3} We now proceed to estimate $\textrm{II}$. \EO{We first extend $\bar{p}(\orsf)$ by zero to $\C$ and then define $\psi = \pe(\orsf) - \bar{p}(\orsf) \in \HL(y^{\alpha},\C)$. We notice that $\psi$ solves
\[
a(\psi,\phi) = (  \tr (\ue(\orsf) - \bar{v}(\orsf)), \tr \phi )_{L^2(\Omega)}  \quad \forall \phi \in \HL(y^{\alpha},\C).
\]
Invoking the trace estimate \eqref{Trace_estimate}, together with the well-posedness of the problem above, the estimate \eqref{estimate_s} and Remark \ref{rem:equivalent}, we conclude
\[
\| \tr \psi \|_{L^2(\Omega)} \lesssim \| \nabla \psi \|_{L^2(y^{\alpha},\C)} \lesssim \| \tr ( \ue(\orsf) - \bar{v}(\orsf))\|_{L^2(\Omega)} .
\]
Therefore, the exponential estimate \eqref{trv-v^Y} yields 
$\| \tr \psi \|_{L^2(\Omega)} \lesssim e^{-\sqrt{\lambda_1} \Y/4} \| \orsf \|_{L^2(\Omega)}$. This result, in conjunction with Steps 1 and 2, yield \eqref{control_exp}.}

\noindent \boxed{4} We extend $\bar{v}(\orsf) \in \HL(y^{\alpha},\C_{\Y})$ by zero to $\C$. Then, $\oue(\ozsf) - \bar{v}(\orsf) \in \HL(y^{\alpha},\C)$ solves
\[
 a(\oue(\ozsf) - \bar{v}(\orsf), \phi) = \left(  \ozsf - \orsf, \tr \phi \right)_{L^2(\Omega)},
 \qquad \forall \phi \in \HL(y^{\alpha},\C).
\]
The well-posedness of the problem above, together with Remark \ref{rem:equivalent} and \eqref{estimate_s}, yield
\[
 \| \nabla( \oue(\ozsf) - \bar{v}(\orsf) )\|_{L^2(y^{\alpha},\C)} \lesssim 
 \| \orsf - \ozsf  \|_{L^2(\Omega)} \lesssim  e^{-\sqrt{\lambda_1}\Y/4}  \| \orsf \|_{L^2(\Omega)},
\]
which implies \eqref{state_exp} and concludes the proof.
\end{proof}

We conclude this section with the following regularity result.

\begin{remark}[regularity of $\orsf$ vs.~$\ozsf$] \rm
\label{rm:regularity_control_r}
\EO{In Lemma~\ref{LM:reg_control} we studied the regularity of $\ozsf$. The techniques of \cite[Remark 4.4]{NOS3} allow us to transfer these results to the solution of the truncated optimal control problem $\orsf$. Similarly, we can derive the regularity results of Remark~\ref{rk:reg_states} for $\tr \bar{v}$ and $\tr \bar{p}$. For brevity we skip the details.}
\end{remark}

\section{A priori error estimates}
\label{sec:apriori}
In this section, we propose and analyze two simple numerical strategies to solve the \emph{fractional optimal control problem} \eqref{Jintro}-\eqref{cc}: a semi-discrete scheme based on the so-called variational approach introduced by Hinze in \cite{Hinze:05} and a fully-discrete scheme which discretizes both the state and control spaces. Before proceeding with the analysis of our method, it is instructive to review the a priori error analysis for the numerical approximation of the state equation \eqref{extension_weak_Y} developed in \cite{NOS}. In an effort to make this contribution self-contained, such results are briefly presented in the following subsection.
\subsection{A finite element method for the state equation}
\label{subsec:state_equation}
In order to study the finite element discretization of problem \eqref{extension_weak_Y}, we must first understand the regularity of the solution $\ue$, since an error estimate for $v$, solution of \eqref{extension_weak_Y}, depends on the regularity of $\ue$; see \cite[\S 4.1]{NOS} and \cite[Remark 4.4]{NOS3} . We recall that \cite[Theorem 2.7]{NOS} reveals that the second order regularity of $\ue$ is significantly worse in 
the extended direction, since it requires a stronger weight, namely $y^{\beta}$ with $\beta > \alpha + 1$:
\begin{align}
    \label{reginx}
  \| \mathcal{L} \ue\|_{L^2(y^{\alpha},\C)} + 
  \| \partial_y \nabla_{x'} \ue \|_{L^2(y^{\alpha},\C)}
  & \lesssim \| \zsf \|_{\Ws}, \\
\label{reginy}
  \| \ue_{yy} \|_{L^2(y^{\beta},\C)} &\lesssim \| \zsf \|_{L^2(\Omega)}.
\end{align}
These estimates suggest that \emph{graded meshes} in the extended variable $y$ play a fundamental role. In fact, since $\ue_{yy} \approx y^{-\alpha -1 }$ as $y \approx 0$ (see \cite[(2.35)]{NOS}), which follows from the behavior of the functions $\psi_k$ defined in \eqref{psik}, we have that anisotropy in the extended variable is fundamental to recover quasi-optimality; see \cite[\S5]{NOS}. This in turn motivates the following construction of a mesh over the truncated cylinder $\C_{\Y}$. Before discussing it, we remark that the regularity estimates \eqref{reginx}-\eqref{reginy} have been also recently established for the solution $v$ to problem \eqref{extension_weak_Y}; see \cite[Remark 4.4]{NOS3}.

To avoid technical difficulties we have assumed that the boundary of $\Omega$ is polygonal. The case of curved boundaries could be handled, for instance, with the methods of \cite{Bernardi}. Let $\T_\Omega = \{K\}$ be a conforming mesh of $\Omega$, where $K \subset \R^n$ is an element that is isoparametrically equivalent either to the unit cube $[0,1]^n$ or the unit simplex in $\R^n$. We denote by $\Tr_\Omega$ the collection of all conforming refinements of an original mesh $\T_{\Omega}^0$. We assume $\Tr_\Omega$ is \emph{shape regular} \cite{CiarletBook,Guermond-Ern}. We also consider a graded partition $\mathcal{I}_\Y$ of the interval $[0,\Y]$ based on intervals $[y_{k},y_{k+1}]$, where
\begin{equation}
\label{graded_mesh}
  y_k = \left( \frac{k}{M}\right)^{\gamma} \Y, \quad k=0,\dots,M,
\end{equation}
and $\gamma > 3/(1-\alpha)=3/(2s) > 1$.
We then construct the mesh $\T_{\Y}$ over $\C_{\Y}$ as the tensor product triangulation 
of $\T_\Omega$ and $\mathcal{I}_\Y$. We denote by $\Tr$ the set of all triangulations of $\C_\Y$ that are obtained with this procedure, and we recall that $\Tr$ satisfies the following regularity assumption (\cite{DL:05,NOS}): there is a constant $\sigma_{\Y}$ such that if $T_1=K_1\times I_1$ and $T_2=K_2\times I_2 \in \T_\Y$ have nonempty intersection, and $h_I = |I|$, then
\begin{equation}
\label{shape_reg_weak}
     h_{I_1} h_{I_2}^{-1} \leq \sigma_{\Y} .
\end{equation}

\begin{remark}[anisotropic estimates]
\rm The weak regularity condition \eqref{shape_reg_weak} allows for a rather general family of anisotropic meshes. Under this assumption weighted and anisotropic interpolation error estimates have been derived in \cite{DL:05,NOS,NOS2}.
\end{remark}

\begin{remark}[$s$-independent mesh grading]
\rm The term $\gamma = \gamma(s)$, which defines the graded mesh $\mathcal{I}_{\Y}$ by \eqref{graded_mesh}, deteriorates as $s$ becomes small because $\gamma > 3/(2s)$. However, a modified mesh grading in the $y$-direction has been proposed in \cite[\S7.3]{CNOS}, which does not change the ratio of the degrees of freedom in $\Omega$ and the extended dimension by more than a constant and provides a uniform bound with respect to $s$.
\end{remark}

For $\T_{\Y} \in \Tr$, we define the finite element space 
\begin{equation}
\label{eq:FESpace}
  \V(\T_\Y) = \left\{
            W \in C^0( \overline{\C_\Y}): W|_T \in \mathcal{P}_1(K) \otimes \mathbb{P}_1(I) \ \forall T \in \T_\Y, \
            W|_{\Gamma_D} = 0
          \right\},
\end{equation}
where $\Gamma_D = \partial_L \C_{\Y} \cup \Omega \times \{ \Y\}$ is the Dirichlet boundary. If the base $K$ of $T = K \times I$ is a simplex, $\mathcal{P}_1(K) = \mathbb{P}_1(K)$: the set of polynomials of degree at most $1$. If $K$ is a cube, $\mathcal{P}_1(K) = \mathbb{Q}_1(K)$ \ie the set of polynomials of degree at most $1$ in each variable.

The Galerkin approximation of \eqref{extension_weak_Y} is given by the 
function $V_{\T_{\Y}} \in \V(\T_{\Y})$ solving
\begin{equation}
\label{harmonic_extension_weak}
  \int_{\C_\Y} y^{\alpha}\nabla V_{\T_{\Y}} \nabla W = 
\langle \rsf, \textrm{tr}_{\Omega} W \rangle_{\Hsd \times \Hs}
  \quad \forall W \in \V(\T_{\Y}).
\end{equation}
Existence and uniqueness of $V_{\T_{\Y}}$ immediately follows from $\V(\T_\Y) \subset \HL(y^{\alpha},\C_\Y)$
and the Lax-Milgram Lemma. 

We define the space $\U(\T_{\Omega})=\tr \V(\T_{\Y})$, which is simply a $\mathcal{P}_1$ finite element space over the mesh $\T_\Omega$. The finite element approximation of $\usf \in \Hs$, solution of \eqref{fractional} with $\rsf$
as a datum, is then given by
\begin{equation}
\label{eq:defofU}
  U_{\T_\Omega} := \tr V_{\T_\Y} \in \U(\T_\Omega ).
\end{equation}

It is trivial to obtain a best approximation result \emph{\`a la} Cea for problem \eqref{harmonic_extension_weak}. This best approximation result reduces the numerical analysis of problem \eqref{harmonic_extension_weak} to a question in approximation theory, which in turn can be answered with the study of piecewise polynomial interpolation in Muckenhoupt weighted Sobolev  spaces; see \cite{NOS,NOS2}. Exploiting the structure of the mesh it is possible to handle anisotropy in the extended variable, construct a quasi-interpolant $I_{\T_\Y} : L^1(\C_\Y) \to \V(\T_\Y)$, and obtain
\begin{align*}
  \| v - I_{\T_\Y} v \|_{L^2(y^\alpha,T)} & \lesssim 
    h_K  \| \nabla_{x'} v\|_{L^2(y^\alpha,S_T)} + h_I \| \partial_y v\|_{L^2(y^\alpha,S_T)}, \\
  \| \partial_{x_j}(v - I_{\T_\Y} v) \|_{L^2(y^\alpha,T)} &\lesssim
    h_K  \| \nabla_{x'} \partial_{x_j} v\|_{L^2(y^\alpha,S_T)} + h_I \| \partial_y \partial_{x_j} v\|_{L^2(y^\alpha,S_T)},
\end{align*}
with $j=1,\ldots,n+1$ and $S_T$ being the patch of $T$; see \cite[Theorems 4.6--4.8]{NOS} and \cite{NOS2} for details. However, since $v_{yy} \approx y^{-\alpha -1 }$ as $y \approx 0$, we realize that $v\notin H^2(y^{\alpha},\C_{\Y})$ and the second estimate is not meaningful for $j=n+1$; see \cite[Remark 4.4]{NOS3}. In view of estimate \eqref{reginy} it is necessary to measure the regularity of $v_{yy}$ with a stronger weight and thus compensate with a graded mesh in the extended dimension. This makes anisotropic estimates essential. 

Notice that $\#\T_{\Y} = M \, \# \T_\Omega$, and that $\# \T_\Omega \approx M^n$ implies $\#\T_\Y \approx M^{n+1}$. Finally, if $\T_\Omega$ is quasi-uniform, we have $h_{\T_{\Omega}} \approx (\# \T_{\Omega})^{-1/n}$. All these considerations allow us to obtain the following result, which follows from \cite[Proposition 4.7]{NOS3}.

\begin{theorem}[a priori error estimate]
\label{TH:fl_error_estimates}
Let $\Omega$ satisfies \eqref{Omega_regular} and let $\T_\Y \in \Tr$ be a tensor product grid, which is quasi-uniform in $\Omega$ and graded in the extended variable so that \eqref{graded_mesh} holds. If $\rsf \in \mathbb{H}^{1-s}(\Omega)$, $\usf$ denotes the solution of \eqref{fractional} with $\rsf$ as a datum, $v$ solves problem \eqref{extension_weak_Y}, $V_{\T_\Y} \in \V(\T_\Y)$ is the Galerkin approximation defined by \eqref{harmonic_extension_weak}, 
and $U_{\T_\Omega} \in \U(\T_\Y)$ is the approximation defined by \eqref{eq:defofU},
then we have
\begin{equation}
\label{l2optimal_ratesinlog}
\| \tr v - U_{\T_\Omega} \|_{L^2(\Omega)} = \| \tr (v - V_{\T_\Y}) \| _{L^2(\Omega)}\lesssim \Y^{2s}(\# \T_{\Y})^{ -\frac{1+s}{n+1}} \|\rsf\|_{\mathbb{H}^{1-s}(\Omega)},
\end{equation}
and
\begin{equation}
\label{l2optimal_rate}
\| \usf - U_{\T_\Omega}  \|_{L^2(\Omega)} \lesssim \Y^{2s}(\# \T_{\Y})^{ -\frac{1+s}{n+1}} \| \rsf \|_{\mathbb{H}^{1-s}(\Omega)}.
\end{equation}
where $\Y \approx |\log(\# \T_{\Y})|$. 
\end{theorem}

\begin{remark}[regularity assumptions]\rm
\label{rk:reg_asump}
As Theorem \ref{TH:fl_error_estimates} indicates, \eqref{l2optimal_ratesinlog} and \eqref{l2optimal_rate} require that $\rsf \in \Ws$ and that the domain $\Omega$ satisfies \eqref{Omega_regular}. If any of these two conditions fail singularities may develop in the direction of the $x'$-variables, whose characterization is an open problem; see \cite[\S~6.3]{NOS} for an illustration. Consequently, quasi-uniform refinement of $\Omega$ would not result in an efficient solution technique and then adaptivity is essential to recover quasi-optimal rates of convergence \cite{CNOS2}.
\end{remark}

\begin{remark}[suboptimal and optimal estimates]
\rm The error estimate \eqref{l2optimal_rate} is optimal in terms of regularity but suboptimal in terms of order. The main ingredients in its derivation are the standard $L^2(\Omega)$-projection, the so-called duality argument \EO{and the regularity estimates \eqref{reginx}-\eqref{reginy} for $v$; see \cite[Remark 4.4]{NOS3}. The role of these regularity estimates can be observed directly from the estimate (4.16) in \cite[Proposition 4.7]{NOS3}, which reads:
\[
\| \tr (v - V_{\T_\Y}) \| _{L^2(\Omega)} \lesssim \Y^{2s}(\# \T_{\Y})^{-\frac{1+s}{n+1}} \mathcal{S}(v),
\]
where $\mathcal{S}(v) = \|\nabla \nabla_{x'} v \|_{L^2(y^{\alpha},\C_{\Y})} +  \|\partial_{yy} v \|_{L^2(y^{\beta},\C_{\Y})}$. Then \eqref{l2optimal_rate} follows by invoking the regularity results of \cite[Remark 4.4]{NOS3}: $\mathcal{S}(v) \lesssim \| \rsf \|_{\Ws}$.} We also remark that \eqref{l2optimal_rate} holds under the anisotropic setting of $\Tr$ given by \eqref{shape_reg_weak}.
\end{remark}

\begin{remark}[Computational complexity]
\rm The cost of solving \eqref{harmonic_extension_weak} is related to $\#\T_\Y$, and not to $\#\T_\Omega$, but the resulting system is sparse. The structure of \eqref{harmonic_extension_weak} is so that fast multilevel solvers can be designed with complexity proportional to $\#\T_\Y (\log(\#\T_\Y))^{1/(n+1)}$ \cite{CNOS}. We also comment that a discretization of the intrinsic integral formulation of the fractional Laplacian result in a dense matrix and involves the development of accurate quadrature formulas for singular integrands; see \cite{HO} for a finite difference approach in one spatial dimension.
\end{remark}

\subsection{The variational approach: a semi-discrete scheme}
\label{subsec:va}
We consider the variational approach introduced and analyzed by Hinze in \cite{Hinze:05}, which only discretizes the state space; the control space $\Zad$ is not discretized. It guarantees conformity since the continuous and discrete admissible sets coincide and induces a discretization of the optimal control by projecting the discrete adjoint state into the admissible control set. Following \cite{Hinze:05}, we consider the following semi-discretized optimal control problem: \EO{ 
$
 \text{min } J( \tr V ,\gsf),
$}
subject to the discrete state equation
\begin{equation}
\label{extension_discrete}
  a_\Y(V,W) = \langle \gsf, \textrm{tr}_{\Omega} W \rangle_{\Hsd \times \Hs}
  \quad \forall W \in \V(\T_{\Y}),
\end{equation}
and the control constraints 
$
 \gsf \in \Zad.
$
For convenience, we will refer to the problem described above as \emph{the semi-discrete optimal control problem}.

We denote by $(\bar{V}, \ogsf) \in \V(\T_\Y) \times \Zad$ the optimal pair solving 
the semi-discrete optimal control problem. Then, by defining
\begin{equation}
\label{Hinze_U}
 \bar{U}:= \tr \bar{V},
\end{equation}
we obtain a semi-discrete approximation $(\bar{U},\ogsf) \in \U(\T_{\Omega}) \times \Zad$ of the optimal pair $(\ousf,\ozsf) \in \Hs \times \Zad$ solving the fractional optimal control problem \eqref{Jintro}-\eqref{cc}.

\begin{remark}[locality] 
\label{rk:locality1}
\rm The main advantage of the semi-discrete control problem is its local nature, thereby mimicking that of the extended optimal control problem.
\end{remark}

In order to study the semi-discrete optimal control problem, we define the control-to-state operator $\mathbf{H}_{\T_\Y}: \Zad \rightarrow \U(\T_\Omega) $, which given a control $\gsf \in \Zad$ associates a unique discrete state $\mathbf{H}_{\T_\Y}\gsf = \tr V (\gsf)$ solving problem \eqref{extension_discrete}. This operator is linear and continuous as a consequence of the Lax-Milgram Lemma. \EO{In addition, it is a self-adjoint operator with respect to the standard $L^2(\Omega)$-inner product.}

We define the optimal adjoint state $\bar{P}= \bar{P}(\ogsf) \in \V(\T_\Y)$ as the solution to
\begin{equation}
\label{adjoint_discrete}
  a_\Y( \bar{P},W) = ( \tr \bar{V} - \usfd , \textrm{tr}_{\Omega} W )_{L^2(\Omega)}
  \quad \forall W \in \V(\T_{\Y}).
\end{equation}

We now state the existence and uniqueness of the optimal control together with the first order optimality conditions for the semi-discrete optimal control problem.

\begin{theorem}[existence, uniqueness and optimality system]
The semi-discrete optimal control problem has a unique optimal solution $(\bar{V}, \ogsf) $ $\in$ $ \V(\T_{\Y})\times \Zad$. The optimality system
\begin{equation}
\begin{dcases}
 \bar{V} = \bar{V}(\ogsf) \in \V(\T_\Y) \textrm{ solution of } \eqref{extension_discrete}, \\
 \bar{P} = \bar{P}(\ogsf) \in \V(\T_\Y) \textrm{ solution of } \eqref{adjoint_discrete}, \\
 \ogsf \in \Zad, \quad (\tr \bar{P} + \mu \ogsf , \gsf - \ogsf )_{L^2(\Omega)} \geq 0 
 \quad \forall \gsf \in \Zad,
\end{dcases}
\label{op_discrete} 
\end{equation}
hold. These conditions are necessary and sufficient
\end{theorem}
\begin{proof}
The proof follows the same arguments employed in the proof of Theorem 
\ref{TH:fractional_control}. For brevity, we skip the details.
\end{proof}

To derive error estimates for our semi-discrete optimal control problem, we rewrite the a priori theory of \S\ref{subsec:state_equation} in terms of the control-to-state operators $\mathbf{H}$ and $\mathbf{H}_{\T_{\Y}}$. Given $\rsf \in \mathbb{H}^{1-s}(\Omega)$, the estimate \eqref{l2optimal_ratesinlog} reads
\begin{equation}
\label{Hl2optimal_rate}
\| (\mathbf{H} - \mathbf{H}_{\T_{\Y}} ) \rsf\|_{L^2(\Omega)} \lesssim 
\Y^{2s} (\# \T_{\Y})^{-\frac{1+s}{n+1}} \| \rsf \|_{\mathbb{H}^{1-s}(\Omega)}.
\end{equation}

The estimate \eqref{Hl2optimal_rate} requires $\orsf \in \Ws$ (see also Theorem \ref{TH:fl_error_estimates} and Remark \ref{rk:reg_asump}). We derive such a regularity result in the following Lemma.

\begin{lemma}[$\Ws$-regularity of control]
\label{LM:control_reg_2}
Let $\orsf \in \Zad$ be the optimal control of the truncated optimal control problem and $\usf_d \in \Ws$. \EO{If, $\asf,\bsf \in H^1(\Omega)$, and, in addition, \eqref{ab_condition} holds} for $s\in(0,\frac{1}{2}]$, then $\orsf \in \mathbb{H}^{1-s}(\Omega)$.
\end{lemma}
\begin{proof}
For $s \in (\frac{1}{2},1)$, we have that $H^{1-s}(\Omega) = H^{1-s}_0(\Omega)$. Then 
\eqref{def:Hs}, Lemma~\ref{LM:reg_control} and Remark \ref{rm:regularity_control_r} yield $\orsf \in \Ws$. When $s \in (0,\frac{1}{2}]$ and \EO{\eqref{ab_condition} is satisfied}, Lemma~\ref{LM:reg_control} and Remark \ref{rm:regularity_control_r}, yields immediately $\ozsf \in H^1_0(\Omega) \subset \mathbb{H}^{1-s}(\Omega)$.
\end{proof}

We now present an a priori error estimate for the semi-discrete optimal control problem. The proof is inspired by the original one introduced by Hinze in \cite{Hinze:05} and it is based on the error estimate \eqref{Hl2optimal_rate}. However, we recall the arguments to verify that they are still valid in the anisotropic framework of \cite{NOS} summarized in \S~\ref{subsec:state_equation}, \EO{and under the regularity properties of the optimal control $\orsf$ dictated by Lemma \ref{LM:control_reg_2}.}

\begin{theorem}[variational approach: error estimate]
\label{TH:variational_error}
Let the pairs $(\bar{v}(\orsf), \orsf)$ and $(\bar{V}(\ogsf), \ogsf)$ be the solutions to the truncated and the semi-discrete optimal control problems, respectively. Then, under the framework of Lemma \ref{LM:control_reg_2}, we have
\begin{equation}
\label{eq:variational_error}
\| \orsf - \ogsf \|_{L^2(\Omega)} \lesssim \Y^{2s}  (\# \T_{\Y})^{ -\frac{1+s}{n+1} } 
\left( \| \orsf \|_{\mathbb{H}^{1-s}(\Omega)} + \| \usfd \|_{\mathbb{H}^{1-s}(\Omega)} \right),
\end{equation}
and
\begin{equation}
\label{eq:variational_errorstate}
\| \tr (\bar{v} - \bar{V}) \|_{L^2(\Omega)} \lesssim \Y^{2s} (\# \T_{\Y})^{ -\frac{1+s}{n+1} } 
( \| \orsf \|_{\mathbb{H}^{1-s}(\Omega)}  
+ \| \usfd \|_{\mathbb{H}^{1-s}(\Omega)} ),
\end{equation}
where $\Y \approx |\log(\# \T_{\Y})|$. 
\end{theorem}
\begin{proof}
Similarly to the proof of Lemma \eqref{LE:exp_convergence}, we set $\rsf = \ogsf \in \Zad$ and $\gsf = \orsf \in \Zad$ in the variational inequalities of the systems \eqref{op_truncated} and \eqref{op_discrete} respectively. Adding the obtained results, we arrive at
\[
\mu \| \orsf - \ogsf \|^2_{L^2(\Omega)} 
\leq ( \tr (\bar{p} -  \bar{P}), \ogsf - \orsf )_{L^2(\Omega)}. 
\]
We now proceed to use the relations
$\tr \bar{p} =  \tr \bar{p}(\orsf) = \mathbf{H} (\mathbf{H}\orsf-\usfd)$ and 
$ \tr \bar{P} = \tr \bar{P}(\ogsf) = \mathbf{H}_{\T_{\Y}} (\mathbf{H}_{\T_{\Y}}\ogsf-\usfd)$, 
to rewrite the inequality above as
\[
\mu \| \orsf - \ogsf \|^2_{L^2(\Omega)} 
\leq ( \mathbf{H}^2 \orsf - \mathbf{H}^2_{\T_{\Y}} \ogsf + (\mathbf{H}_{\T_{\Y}} - \mathbf{H}) \usfd, \ogsf - \orsf )_{L^2(\Omega)}
\]
which, by adding and subtracting the term $\mathbf{H}_{\T_{\Y}} \EO{\mathbf{H}} \orsf$, yields
\[
 \mu \| \orsf - \ogsf \|^2_{L^2(\Omega)} \leq \big( (\mathbf{H}   - \mathbf{H}_{\T_{\Y}}) \mathbf{H} \orsf + \mathbf{H}_{\T_{\Y}}\mathbf{H}\orsf - \mathbf{H}^2_{\T_{\Y}} \ogsf + (\mathbf{H}_{\T_{\Y}} - \mathbf{H}) \usfd, \ogsf - \orsf \big)_{L^2(\Omega)}.
\]
We now add and subtract the term \EO{$\mathbf{H}_{\T_{\Y}}^2 \orsf$} to arrive at
\begin{multline}
\nonumber
\mu \| \orsf - \ogsf \|^2_{L^2(\Omega)}  \leq 
\EO{((\mathbf{H} -\mathbf{H}_{\T_{\Y}}) \mathbf{H} \orsf, \ogsf - \orsf)_{L^2(\Omega)} +  
(\mathbf{H}_{\T_\Y} ( \mathbf{H} -\mathbf{H}_{\T_\Y} ) \orsf, \ogsf - \orsf)_{L^2(\Omega)}}
\\
\nonumber \EO{
+ (\mathbf{H}_{\T_\Y} ^2 ( \orsf - \ogsf), \ogsf - \orsf)_{L^2(\Omega)}
+ ((\mathbf{H}_{\T_{\Y}} - \mathbf{H}) \usfd, \ogsf - \orsf )_{L^2(\Omega)} = \textrm{I} + \textrm{II} + \textrm{III} + \textrm{IV}.}
\end{multline}
\EO{We estimate the term \textrm{I} as follows: 
\begin{align*} 
| \textrm{I} | \leq  \| (\mathbf{H} -\mathbf{H}_{\T_{\Y}}) \mathbf{H} \orsf\|_{L^2(\Omega)} \| \ogsf - \orsf \|_{L^2(\Omega)}  \lesssim \Y^{2s}(\# \T_{\Y})^{-\frac{1+s}{n+1}} \| \mathbf{H} \orsf \|_{\mathbb{H}^{1-s}(\Omega)} \| \ogsf - \orsf \|_{L^2(\Omega)}
\end{align*}
where we have used the approximation property \eqref{Hl2optimal_rate}. Now, since $\mathbf{H} \orsf = \tr v(\orsf)$ and $\orsf \in H^1(\Omega) \cap \mathbb{H}^{1-s}(\Omega)$, the arguments developed in \cite[Remark 4.4]{NOS3}, in conjunction with Remark \ref{rk:reg_states}, yield $\| \tr v(\orsf) \|_{\Ws} \lesssim \| \orsf  \|_{\Ws}$. The estimate for terms \textrm{II} and \textrm{IV} follow exactly the same arguments by using the continuity of $\mathbf{H}_{\T_{\Y}}$.}

\EO{The desired estimate \eqref{eq:variational_error} is then a consequence of the derived estimates in conjunction with the fact that \textrm{III} $ \leq 0$.}

Finally, we prove the estimate \eqref{eq:variational_errorstate}. Since $\bar{v} = \bar{v}(\orsf)$ and $\tr \bar{V} = \tr \bar{V}(\ogsf)$, we conclude that
that
\begin{align}
\nonumber
\| \tr (\bar{v} - \bar{V}) \|_{L^2(\Omega)} & = 
\|  \mathbf{H} \orsf  - \mathbf{H}_{\T_{\Y}} \ogsf \|_{L^2(\Omega)} 
\leq \|  \EO{(\mathbf{H} - \mathbf{H}_{\T_\Y})\orsf}  \|_{L^2(\Omega)} 
+ \|  \EO{\mathbf{H}_{\T_\Y} (\orsf - \ogsf)} \|_{L^2(\Omega)}
\\
\label{aux:hinze2}
&
\lesssim \Y^{2s}(\# \T_{\Y})^{-\frac{1+s}{n+1}}\left( \| \orsf \|_{\mathbb{H}^{1-s}(\Omega)} + \|\usf_d \|_{\Ws}\right),
\end{align}
where we have used the estimates \eqref{Hl2optimal_rate} and \eqref{eq:variational_error}, and the continuity of $\mathbf{H}_{\T_{\Y}}$. This yields \eqref{eq:variational_errorstate} and concludes the proof.
\end{proof}

\begin{remark}[variational approach: advantages and disadvantages]
\rm The key advantage of the variational approach is obtaining an optimal quadratic rate of convergence for the control \cite[Theorem 2.4]{Hinze:05}. However, given \eqref{l2optimal_ratesinlog}, in our case it allows us to derive \eqref{eq:variational_error}, which is suboptimal in terms of order but optimal in terms of regularity. From an implementation perspective this technique may lead to a control which is not discrete in the current mesh and thus requires an independent mesh.
\end{remark}

\begin{remark}[anisotropic meshes]
\rm Examining the proof of Theorem \ref{TH:variational_error}, we realize that the critical steps, where the anisotropy of the mesh $\T_{\Y}$ is needed, are both, \EO{at estimating the term \textrm{I} and in \eqref{aux:hinze2}.} The analysis developed in \cite{Hinze:05} allows the use anisotropic meshes through the estimate \eqref{Hl2optimal_rate}. This fact can be observed in Theorem \ref{TH:variational_error} and has also been exploited to address control problems on nonconvex domains; see \cite[\S~6]{APR:12}, \cite[\S~3]{ARD:09} and \cite[\S~4]{ARW:07}.
\end{remark}

We conclude this subsection with the
following  consequence of Theorem 
\ref{TH:variational_error}.

\begin{corollary}[fractional control problem: error estimate]
\label{CR:variational_error2}
Let $(\bar{V},\ogsf) $ $\in \V(\T_\Y) \times \Zad$ solve the semi-discrete optimal control problem and $\bar{U} \in \U(\T_{\Omega})$ be defined as in \eqref{Hinze_U}. Then, under the framework of Lemma \eqref{LM:control_reg_2}, we have
\begin{equation}
\label{Hinze1}
  \| \ozsf - \ogsf \|_{L^2(\Omega)} \lesssim \Y^{2s} (\# \T_{\Y})^{ -\frac{1+s}{n+1} } 
( \| \orsf \|_{\mathbb{H}^{1-s}(\Omega)} 
+ \| \usfd \|_{\mathbb{H}^{1-s}(\Omega)}),
\end{equation} 
and
\begin{equation}
\label{Hinze2}
\| \ousf - \bar{U} \|_{L^2(\Omega)} \lesssim  \Y^{2s}(\# \T_{\Y})^{- \frac{1+s}{n+1} } 
\left( \| \orsf \|_{\mathbb{H}^{1-s}(\Omega)} + \| \usfd \|_{\mathbb{H}^{1-s}(\Omega)} \right),
\end{equation}
where $\Y \approx |\log(\# \T_{\Y})|$ and $(\bar{\usf},\ozsf) \in \Hs \times \Zad$ solves \eqref{Jintro}-\eqref{cc}.
\end{corollary}
\begin{proof}
We recall that $(\oue,\ozsf) \in \HL(y^{\alpha},\C) \times \Zad$ and $(\bar{v},\orsf) \in \HL(y^{\alpha},\C_{\Y}) \times \Zad$ solve the extended and truncated optimal control problems, respectively. Then, triangle inequality in conjunction with 
Lemma \ref{LE:exp_convergence}, Lemma \ref{LM:control_reg_2} and Theorem \ref{TH:variational_error} yield
\begin{align*}
\| \ozsf - \ogsf  \|_{L^2(\Omega)} & \leq 
 \| \ozsf - \orsf  \|_{L^2(\Omega)} +  \| \orsf - \ogsf  \|_{L^2(\Omega)} 
\\
& \lesssim \left( e^{-\sqrt{\lambda_1} \Y/4}  + \Y^{2s} (\# \T_{\Y})^{-\frac{1+s}{n+1}} \right)
\left( \| \orsf \|_{\mathbb{H}^{1-s}(\Omega)}  +  \| \usf_d \|_{\mathbb{H}^{1-s}(\Omega)}\right)
\\
& \lesssim  |\log(\# \T_{\Y})|^{2s}(\# \T_{\Y})^{-\frac{1+s}{n+1}} \left( \| \orsf \|_{\mathbb{H}^{1-s}(\Omega)}  +  
\| \usf_d \|_{\mathbb{H}^{1-s}(\Omega)}\right),
\end{align*}
which is exactly the desired estimate \eqref{Hinze1}. In the last inequality above
we have used $\Y \approx \log(\# (\T_\Y))$; see \cite[Remark 5.5]{NOS}. 
To derive \eqref{Hinze2}, we proceed as follows:
\begin{align*}
\| \ousf - \bar{U} \|_{L^2(\Omega)} 
&\leq 
\| \ousf - \tr \bar{v}  \|_{L^2(\Omega)} +  \| \tr \bar{v} - \bar{U} \|_{L^2(\Omega)} 
\\
&\lesssim  |\log(\# \T_{\Y})|^{2s}(\# \T_{\Y})^{-\frac{1+s}{n+1}} 
\left( \| \orsf \|_{\mathbb{H}^{1-s}(\Omega)}   + \| \usf_d \|_{\mathbb{H}^{1-s}(\Omega)}\right),
\end{align*}
where we have used \eqref{state_exp} and \eqref{eq:variational_errorstate}. This concludes the proof.
\end{proof}

\subsection{A fully discrete scheme}
\label{subsec:fd}

The goal of this subsection is to introduce and analyze a fully-discrete scheme to solve the fractional optimal control problem \eqref{Jintro}-\eqref{cc}. We propose the following fully-discrete approximation of the truncated control problem analyzed in \S\ref{sec:control_truncated}:
$
 \EO{\text{min } J(\tr V , Z)},  
$
subject to the discrete state equation
\begin{equation}
\label{fd_a}
  a_\Y(V,W) =  \langle Z, \tr W \rangle
    \quad \forall W \in \V(\T_{\Y}),
\end{equation}
and the discrete control constraints
$
Z \in \mathbb{Z}_{ad}(\T_{\Omega}).
$
The functional $J$ and the discrete space $\V(\T_{\Y})$ are defined by \eqref{functional} and \eqref{eq:FESpace}, respectively. In addition,
\begin{equation*}
\mathbb{Z}_{ad}(\T_{\Omega}) = \Zad \cap \left\{
            Z \in L^{\infty}( \Omega ): Z|_K \in \mathbb{P}_0(K) \quad \forall K \in \T_\Omega \right\}
\end{equation*}
denotes the discrete and admissible set of controls, which is discretized by piecewise constant functions. \EO{To simplify the exposition, in what follows we assume that, in the definition of $\Zad$, given by \eqref{ac}, $\asf$ and $\bsf$ are constant. For convenience, we will refer to the problem previously defined as the \emph{fully-discrete optimal control problem.}}

We denote by $(\bar{V}, \bar{Z}) \in \V(\T_\Y) \times  \mathbb{Z}_{ad}(\T_{\Omega})$ 
the optimal pair solving the fully-discrete optimal control problem. Then, by setting
\begin{equation}
\label{Hinze_Ufd}
 \bar{U}:= \tr \bar{V},
\end{equation}
we obtain a fully-discrete approximation $(\bar{U},\bar{Z}) \in  \U(\T_{\Omega}) \times 
\mathbb{Z}_{ad}(\T_{\Omega})$ of the optimal pair $(\ousf,\ozsf) \in \Hs \times \Zad$ solving the fractional optimal control problem \eqref{Jintro}-\eqref{cc}.

\begin{remark}[locality] 
\rm 
The main advantage of the fully-discrete optimal control problem  
is that involves the local problem \eqref{fd_a} as state equation.
\end{remark}

We define the discrete control-to-state operator
$\mathbf{H}_{\T_\Y}: \mathbb{Z}_{ad}(\T_{\Omega}) \rightarrow \U(\T_\Y) $, which given a discrete control $Z\in \mathbb{Z}_{ad}(\T_{\Omega})$ associates a unique discrete state $\mathbf{H}_{\T_\Y}Z = \tr V (Z) $ solving the discrete problem \eqref{fd_a}. 

We define the optimal adjoint state 
$\bar{P}= \bar{P}(\bar{Z}) \in \V(\T_\Y)$ to be the solution to
\begin{equation}
\label{fd_adjoint}
a_\Y( \bar{P},W) = ( \tr \bar{V} - \usfd , \textrm{tr}_{\Omega} W )_{L^2(\Omega)} \quad \forall W \in \V(\T_{\Y}).
\end{equation}

We present the following result, which follows along the same lines as the proof of Theorem \ref{TH:fractional_control}. For brevity, we skip the details.
\begin{theorem}[existence, uniqueness and optimality system]
The fully-discrete optimal control problem has a unique optimal solution 
$(\bar{V}, \bar{Z}) $ $\in$ $ \V(\T_{\Y})\times \mathbb{Z}_{ad}(\T_{\Omega})$. 
The optimality system
\begin{equation}
\begin{dcases}
 \bar{V} = \bar{V}(\bar{Z}) \in \V(\T_\Y) \textrm{ solution of } \eqref{fd_a}, \\
 \bar{P} = \bar{P}(\bar{Z}) \in \V(\T_\Y) \textrm{ solution of } \eqref{fd_adjoint}, \\
 \bar{Z} \in \mathbb{Z}_{ad}(\T_{\Omega}), \quad 
 (\tr \bar{P} + \mu \bar{Z}, Z- \bar{Z})_{L^2(\Omega)} \geq 0 
 \quad \forall Z \in \mathbb{Z}_{ad}(\T_{\Omega}),
\end{dcases}
\label{fd_op}
\end{equation}
hold. These conditions are necessary and sufficient.
\end{theorem}

To derive a priori error estimates for the fully-discrete optimal control problem, we recall the $L^2$-orthogonal projection operator $\Pi_{\T_{\Omega}}: L^2(\Omega) \rightarrow \mathbb{P}_{0}(\T_{\Omega})$ defined by 
\begin{equation}
\label{o_p}
(\rsf - \Pi_{\T_{\Omega}}\rsf , Z ) = 0 \qquad \forall Z \in  \mathbb{Z}_{ad}(\T_{\Omega});
\end{equation}
see \cite{CiarletBook,Guermond-Ern}. The space $\mathbb{P}_{0}(\T_{\Omega})$ denote the space of piecewise constants over the mesh $\T_{\Omega}$. We also recall the following properties of the operator $\Pi_{\T_{\Omega}}$: for all $\rsf \in L^2(\Omega)$, we have
$
\|\Pi_{\T_{\Omega}}\rsf \|_{L^2(\Omega)} \lesssim
 \|\rsf \|_{L^2(\Omega)}.
$
In addition, if $\rsf \in H^{1}(\Omega)$, we have
\begin{equation}
\label{o_p:aprox}
\| \rsf - \Pi_{\T_{\Omega}}\rsf \|_{L^2(\Omega)} \lesssim h_{\T_{\Omega}}\|\rsf \|_{H^1(\Omega)},
\end{equation}
where $h_{\T_{\Omega}}$ denotes the mesh-size of $\T_{\Omega}$; see \cite[Lemma 1.131 and Proposition 1.134]{Guermond-Ern}. Moreover, given $\rsf \in L^2(\Omega)$, \eqref{o_p} immediately yields
$
 \Pi_{\T_{\Omega}}\rsf |_K = (1/|K|)\int_{K} \rsf.
$
Consequently, since $\asf(x') \equiv \asf$ and $\bsf(x') \equiv \bsf$ for all $x' \in \Omega$, we conclude that $\Pi_{\T_{\Omega}}\rsf \in \mathbb{Z}_{ad}(\T_{\Omega})$, and then $\Pi_{\T_{\Omega}}: L^2(\Omega) \rightarrow \mathbb{Z}_{ad}(\T_{\Omega})$ is well defined.

\EO{Inspired by \cite{MV:08}, we now introduce two auxiliary problems. The first one reads: Find $Q \in \mathbb{V}(\T_{\Y})$ such that
\begin{equation}
\label{eq:aux1}
a_\Y( Q ,W) = ( \tr V(\orsf) - \usfd , \textrm{tr}_{\Omega} W )_{L^2(\Omega)} \quad \forall W \in \mathbb{V}(\T_{\Y}).
\end{equation}
The  second  one  is: Find $R \in \mathbb{V}(\T_{\Y})$ such that
\begin{equation}
\label{eq:aux2}
a_\Y( R ,W) = ( \tr \bar{v} - \usfd , \textrm{tr}_{\Omega} W )_{L^2(\Omega)} \quad \forall W \in \mathbb{V}(\T_{\Y}).
\end{equation}}
We now derive error estimates for the fully-discrete optimal control problem.

\begin{theorem}[fully discrete scheme: error estimate]
\label{TH:fd_error}
If $\usf_d \in \Ws$, \EO{\eqref{ab_condition} holds for $s \in (0,\tfrac{1}{2}]$,} and $(\bar{v}(\orsf), \orsf)$ 
and $(\bar{V}(\bar{Z}), \bar{Z})$ solve the truncated and the fully-discrete optimal control problems, respectively, then 
\begin{equation}
\label{eq:fd_error}
\| \orsf - \bar{Z} \|_{L^2(\Omega)} \lesssim 
\Y^{2s}(\# \T_{\Y} )^{-\frac{1}{n+1}}
\left( \| \orsf \|_{H^1(\Omega)} + \| \usfd \|_{\Ws} \right)
\end{equation}
and
\begin{equation}
\label{eq:fd_errorstate}
 \| \tr (\bar{v} - \bar{V}) \|_{\Hs} \lesssim \Y^{2s}(\# \T_{\Y} )^{-\frac{1}{n+1}}  
\left(  \| \orsf \|_{H^1(\Omega)} + \| \usfd \|_{\Ws} \right). 
\end{equation}
where $\Y \approx |\log(\# \T_{\Y})|$. 
\end{theorem}
\begin{proof} We proceed in 5 steps.

\noindent \boxed{1} Since $\dZad \subset \Zad$, we set $\rsf = \bar{Z}$ in 
the variational inequality of 
\eqref{op_truncated} to write
\[
 (\tr \bar{p} + \mu \orsf , \bar{Z} - \orsf )_{L^2(\Omega)} \geq 0.
\]
On the other hand, setting $ Z = \Pi_{\T_{\Omega}} \orsf \in \dZad$, with $\Pi_{\T_{\Omega}}$ defined by \eqref{o_p}, in the variational inequality of \eqref{fd_op}, and adding and subtracting $\orsf$, we derive
\[
(\tr \bar{P} + \mu \bar{Z}, \Pi_{\T_{\Omega}} \orsf - \orsf )_{L^2(\Omega)} 
+
(\tr \bar{P} + \mu \bar{Z}, \orsf - \bar{Z} )_{L^2(\Omega)} 
\geq 0.
\]
Consequently, adding the derived expressions we arrive at
\[
(\tr (\bar{p} - \bar{P}) + \mu (\orsf - \bar{Z}), \bar{Z} - \orsf  )_{L^2(\Omega)} 
+
(\tr \bar{P} + \mu \bar{Z}, \Pi_{\T_{\Omega}} \orsf - \orsf )_{L^2(\Omega)} \geq 0, 
\]
and then
\[
\mu \| \orsf - \bar{Z} \|^2_{L^2(\Omega)} \leq 
(\tr (\bar{p} - \bar{P}), \bar{Z} - \orsf  )_{L^2(\Omega)} 
+
(\tr \bar{P} + \mu \bar{Z}, \Pi_{\T_{\Omega}} \orsf - \orsf )_{L^2(\Omega)} = 
\textrm{I} + \textrm{II}.
\]

\noindent \boxed{2} The estimate for the term \textrm{I} follows immediately as a consequence of the arguments employed in the proof of Theorem \ref{TH:variational_error}, \EO{which only rely on the regularity $\orsf \in \Ws$ given in Lemma \ref{LM:control_reg_2} and the fact that $\usf_d \in \Ws$. To be precise, we have}
\[
 | \textrm{I} | \lesssim 
\Y^{2s} (\# \T_{\Y})^{-\frac{1+s}{n+1}} \left( \| \orsf \|_{\Ws} + 
\| \usfd \|_{\Ws} \right) \| \orsf - \bar{Z}  \|_{L^2(\Omega)}.
\]

\noindent \boxed{3} We estimate the term \textrm{II} \EO{by using the solutions to the problems \eqref{eq:aux1}
and \eqref{eq:aux2}:
\begin{multline*}
\textrm{II}  = ( \tr \bar{P} + \mu \bar{Z}, \Pi_{\T_{\Omega}}  \orsf - \orsf )_{L^2(\Omega)} 
=
( \tr \bar{p} + \mu \orsf, \Pi_{\T_{\Omega}}  \orsf - \orsf )_{L^2(\Omega)} + \mu (\bar{Z} - \orsf, \Pi_{\T_{\Omega}}  \orsf - \orsf)
\\
+ ( \tr (\bar{P}- Q), \Pi_{\T_{\Omega}}  \orsf - \orsf )_{L^2(\Omega)} 
+ ( \tr (Q \pm R - \bar{p} ), \Pi_{\T_{\Omega}}  \orsf - \orsf )_{L^2(\Omega)}
= \textrm{II}_1 + \textrm{II}_2 + \textrm{II}_3 + \textrm{II}_4.
\end{multline*}
Invoking the definition \eqref{o_p} of $\Pi_{\T_{\Omega}}$, we arrive at
\[
 \textrm{II}_1 = ( \tr \bar{p} + \mu \orsf + \Pi_{\T_{\Omega}}(\tr \bar{p} + \mu \orsf), \Pi_{\T_{\Omega}}  \orsf - \orsf )_{L^2(\Omega)},
\]
which by using \eqref{o_p:aprox}, yields $|\textrm{II}_1| \lesssim h^2_{\T_{\Omega}} \| \tr \bar{p} (\orsf) + \mu \orsf \|_{H^1(\Omega)} \| \orsf \|_{H^1(\Omega)}$. Remark \ref{rm:regularity_control_r} guarantees that both $\orsf$ and $\tr \bar{p}$ belong to $\in H^1(\Omega)$. The term $\textrm{II}_2$ is controlled by a trivial application of the Cauchy-Schwarz inequality. To estimate $\textrm{II}_3$, we invoke the stability of the discrete problem \eqref{fd_adjoint} and the estimate \eqref{o_p:aprox} to conclude
\[
|\textrm{II}_3| \lesssim h_{\T_{\Omega}} \| \tr (\bar{V}- V(\orsf)) \|_{L^2(\Omega)} \| \orsf\|_{H^1(\Omega)} 
\lesssim h_{\T_\Omega} \| \bar{Z} - \orsf \|_{L^2(\Omega)} \| \orsf\|_{H^1(\Omega)},  
\]
where in the last inequality we used the discrete stability of \eqref{fd_a}. The estimate for the term $R-\bar{p}$ in $\textrm{II}_4$ follows directly from Theorem~\ref{TH:fl_error_estimates}. In fact,
\[
\| \tr (R - \bar{p} ) \|_{L^2(\Omega)} \lesssim \Y^{2s} (\# \T_{\Y})^{-\frac{1+s}{n+1}} \| \tr \bar{v} - \usf_d\|_{\Ws}.
\]
Remark \ref{rm:regularity_control_r} yields $\tr \bar{v} \in \Ws$. The remainder term $Q-R$ is controlled by using the discrete stability of problem \eqref{eq:aux2} together with Theorem~\ref{TH:fl_error_estimates}:
\[
\| \tr (Q - R) \|_{ L^2(\Omega) }  
\lesssim \| \tr(V(\orsf) - \bar{v}) \|_{L^2(\Omega)} \lesssim \Y^{2s}
(\# \T_{\Y})^{-\frac{1+s}{n+1}} \| \orsf \|_{\Ws}.
\]
These estimates yields $|\textrm{II}_4| \lesssim \Y^{2s}
(\# \T_{\Y})^{-\frac{2+s}{n+1}}(\| \orsf\|_{\Ws} + \| \usf_d \|_{\Ws}) \| \orsf\|_{H^1(\Omega)}$.
}

\noindent \boxed{4} \EO{The estimates derived in Step 3, in conjunction with appropriate applications of the Young's inequality and the bound for \textrm{I} obtained in Step 2, yield the desired estimate \eqref{eq:fd_error}.}

\noindent \boxed{5} Finally, the estimate \eqref{eq:fd_errorstate} follows easily. In fact, since $\tr \bar{v} = \tr \bar{v}(\orsf) = \mathbf{H}\orsf$ and $\tr \bar{V} = \tr \bar{V}(\bar{Z}) = \mathbf{H}_{\T_{\Y}} \bar{Z}$, we conclude
that
\begin{align*}
\| \tr (\bar{v} -  \bar{V})\|_{\Hs} & = \|  \mathbf{H} \orsf  - \mathbf{H}_{\T_{\Y}} \bar{Z} \|_{\Hs} 
\\
& \leq  \|  (\mathbf{H} - \mathbf{H}_{\T_{\Y}}) \orsf \|_{\Hs}
+ \|  \mathbf{H}_{\T_{\Y}} (\orsf  -  \bar{Z}) \|_{\Hs},
\end{align*}
which, as a consequence of the fact that $\orsf \in \Ws$, \cite[Remark 5.6]{NOS}, the continuity of $\mathbf{H}_{\T_{\Y}}$ and \eqref{eq:fd_error} yield \eqref{eq:fd_errorstate}. This concludes the proof.
\end{proof}

We now present the following consequence of Theorem \ref{TH:fd_error}.
\begin{corollary}[fractional control problem: error estimate]
\label{CR:fd}
Let $(\bar{V},\bar{Z}) $ $\in \V(\T_\Y) \times \Zad$ solves the fully-discrete control problem and $\bar{U} \in \U(\T_{\Omega})$ be defined as in \eqref{Hinze_Ufd}. If $\usf_d \in \Ws$, \EO{and $\asf$ and $\bsf$ satisfy \eqref{ab_condition} for $s \in (0,\tfrac{1}{2}]$}, then we have 
\begin{equation}
\label{fd2}
  \| \ozsf - \bar{Z} \|_{L^2(\Omega)} \lesssim  |\log (\# \T_{\Y})|^{2s}(\# \T_{\Y})^{\frac{-1}{n+1}} 
\left( \| \orsf \|_{H^1(\Omega)} + \| \usfd \|_{\Ws} \right),
\end{equation}
and
\begin{equation}
\label{fd1}
\| \ousf - \bar{U} \|_{\Hs} \lesssim  |\log (\# \T_{\Y})|^{2s}(\# \T_{\Y})^{\frac{-1}{n+1}} 
\left( \| \orsf \|_{H^1(\Omega)} + \| \usfd \|_{\Ws}  \right).
\end{equation} 
\end{corollary}
\begin{proof}
We recall that $(\oue,\ozsf) \in \HL(y^{\alpha},\C) \times \Zad$ and $(\bar{v},\orsf) \in \HL(y^{\alpha},\C_{\Y}) \times \Zad$ solve the extended and truncated optimal control problems, respectively. Then, Lemma \ref{LE:exp_convergence} and Theorem \ref{TH:fd_error} imply
\begin{align*}
\| \ozsf - \bar{Z} \|_{L^2(\Omega)} & \leq 
 \| \ozsf - \orsf  \|_{L^2(\Omega)} +  \| \orsf - \bar{Z}\|_{L^2(\Omega)} 
\\
& \lesssim \left( e^{-\sqrt{\lambda_1} \Y/4} 
+  (\# \T_{\Y})^{-\frac{1}{n+1}} \right) \left( \| \orsf \|_{H^1(\Omega)} + \| \usfd \|_{\Ws} \right)
\\
& \lesssim |\log(\# \T_{\Y})|^{2s}(\# \T_{\Y})^{-\frac{1}{n+1}} \left( \| \orsf \|_{H^1(\Omega)} + \| \usfd \|_{\Ws} \right),
\end{align*}
where we have used that $\Y \approx \log(\# (\T_\Y))$; see \cite[Remark 5.5]{NOS} for details. This gives the desired estimate \eqref{fd2}. In order to derive \eqref{fd1}, we proceed as follows:
\begin{align*}
\| \ousf - \bar{U} \|_{\Hs} 
 & \leq 
\| \ousf - \tr \bar{v}  \|_{\Hs} +  \| \tr \bar{v} - \bar{U} \|_{\Hs} 
\\
& \lesssim |\log(\# \T_{\Y})|^{2s}(\# \T_{\Y})^{-\frac{1}{n+1}}
\left( \| \orsf \|_{H^1(\Omega)} + \| \usfd \|_{\Ws} \right),
\end{align*}
where we have used \eqref{state_exp} and \eqref{eq:fd_errorstate}. This concludes the proof.
\end{proof}
\section{Numerical Experiments}
\label{sec:numerics}

In this section, we illustrate the performance of the fully-discrete scheme proposed and analyzed in \S \ref{subsec:fd}  
approximating the fractional optimal control problem \eqref{Jintro}-\eqref{cc}, and the sharpness of the error estimates derived in Theorem \ref{TH:fd_error} and Corollary \ref{CR:fd}.

\subsection{Implementation}
The implementation has been carried out within
the MATLAB$^\copyright$ software library {\it{i}}FEM~\cite{chen2009ifem}.
The stiffness matrices of the discrete system \eqref{fd_op} are assembled exactly, and
the respective forcing boundary term are computed by a quadrature formula which is exact for polynomials of degree $4$. 
The resulting linear system is solved by using the built-in \emph{direct solver} of MATLAB$^\copyright$.
More efficient techniques for preconditioning are currently under investigation.
To solve the minimization problem, we use the gradient based minimization algorithm \emph{fmincon} of MATLAB$^\copyright$. The optimization algorithm stops when the gradient of the cost function is less than or equal to $10^{-8}$.

We now proceed to derive an exact solution to the fractional optimal control problem 
\eqref{Jintro}-\eqref{cc}. To do this, let $n = 2$, $\mu =1$, $\Omega = (0,1)^2$, and 
$c(x') \equiv 0$ and  $A(x') \equiv 1$ in \eqref{second_order}. Under this setting,
the eigenvalues and eigenfunctions of $\mathcal{L}$ are:
\[
\lambda_{k,l} = \pi^2 (k^2 + l^2), \quad \varphi_{k,l}(x_1,x_2) = \sin(k \pi x_1) \sin(l\pi x_2)  
\quad k, l \in \mathbb{N}.
\]
Let $\ousf$ be the solution to   
$
\mathcal{L}^s \ousf = \fsf + \ozsf
$
in $\Omega$, $\usf = 0$  on  $\partial \Omega$,
which is a modification of problem \eqref{fractional}, since we added 
the forcing term $\fsf$. If $\fsf = \lambda_{2,2}^s \sin(2 \pi x_1) \sin(2\pi x_2) - \ozsf$, 
then by \eqref{def:second_frac}, we have $\ousf = \sin(2 \pi x_1) \sin(2\pi x_2)$. Now, we
set $\opsf = -\mu \sin(2 \pi x_1) \sin(2 \pi x_2)$, which by invoking Definition 
\ref{def:fractional_adjoint}, yields
$
\usf_d = (1 + \mu \lambda_{2,2}^s) \sin(2 \pi x_1) \sin(2\pi x_2) .
$
The projection formula \eqref{projection_formula} allows to write
$\ozsf = \min \left\{ \bsf , \max \left\{ \asf , - \opsf/\mu \right\} \right\}$. 
Finally, we set $\asf = 0$ and $\bsf = 0.5$, and we see easily that, for any $s \in (0,1)$, 
$\ozsf \in H_0^1(\Omega) \subset \Ws$. 

Since we have an exact solution $(\ousf,\ozsf)$ to \eqref{Jintro}-\eqref{cc}, we compute
the error $\| \ousf - \bar{U} \|_{\Hs}$ by using \eqref{Trace_estimate} and computing 
$\| \nabla\left( \bar{\ue} - \bar{V} \right) \|_{L^2(y^{\alpha},\C)}$
as follows:
\[
\| \nabla\left( \bar{\ue} - \bar{V} \right) \|_{L^2(y^{\alpha},\C)}^2 
= d_s \int_\Omega (f+\ozsf) \tr \left(  \bar{\ue} - \bar{V} \right),
\] 
which follows from Galerkin orthogonality. Thus, we avoid
evaluating the weight $y^{\alpha}$
and reduce the computational cost.
The right hand side of the equation above is computed by a quadrature formula
which is exact for polynomials of degree 7. 

\subsection{Uniform refinement versus anisotropic refinement}
\label{s:unif}
At the level of solving the state equation \eqref{alpha_harm_L_weak},
a competitive performance of anisotropic over quasi-uniform refinement is discussed in \cite{CNOS2,NOS}.
Here we explore the advantages of using the anisotropic refinement
developed in \S\ref{sec:apriori}
when solving the fully-discrete problem proposed in \S\ref{subsec:fd}.
Table~\ref{t:unif_rate_s02} shows the error in the control and the state for 
uniform (un) and anisotropic (an) refinement for $s = 0.05$. \#DOFs denotes the degrees of 
freedom of $\T_{\Y}$. 
\begin{table}[h!]
\centering
\begin{tabular}{|c|c|c|c|c|}\hline 
 {\footnotesize \#DOFs}  &    {\footnotesize $\| \ozsf - \bar{Z} \|_{L^2(\Omega)}$ (un)} &  {\footnotesize $\| \ozsf - \bar{Z} \|_{L^2(\Omega)}$ (an)} &  {\footnotesize $\|  \ousf - \bar{U} \|_{\Hs}$ (un)} &  {\footnotesize $\|  \ousf - \bar{U} \|_{\Hs}$ (an)} \\ \hline 
  3146 &  1.46088e-01 & 5.84167e-02 & 1.50840e-01 & 8.83235e-02 \\
 10496 &  1.24415e-01 & 4.25698e-02 & 1.51756e-01 & 6.49159e-02 \\
 25137 &  1.11969e-01 & 3.08367e-02 & 1.50680e-01 & 5.04449e-02 \\
 49348 &  1.04350e-01 & 2.54473e-02 & 1.49425e-01 & 4.07946e-02 \\
 85529 &  9.82338e-02 & 2.09237e-02 & 1.48262e-01 & 3.42406e-02 \\
137376 &  9.41058e-02 & 1.81829e-02 & 1.47146e-01 & 2.93122e-02\\ \hline
\end{tabular}
\caption{\label{t:unif_rate_s02}Experimental errors of the fully-discrete scheme studied in \S \ref{subsec:fd}
on uniform (un) and anisotropic (an) refinement. A competitive performance of anisotropic refinement is observed.}
\end{table}
Clearly, the errors obtained with anisotropic refinement are almost an order 
in magnitude smaller than the corresponding errors due to uniform refinement. 
In addition, Figure~\ref{f:unif_ani_s005} shows that the anisotropic refinement 
leads to quasi-optimal rates of convergence for the optimal variables,
thus verifying Theorem \ref{TH:fd_error} and Corollary \ref{CR:fd}.
Uniform refinement produces suboptimal rates of convergence.

\begin{figure}[h!]
\centering
\includegraphics[width=0.48\textwidth]{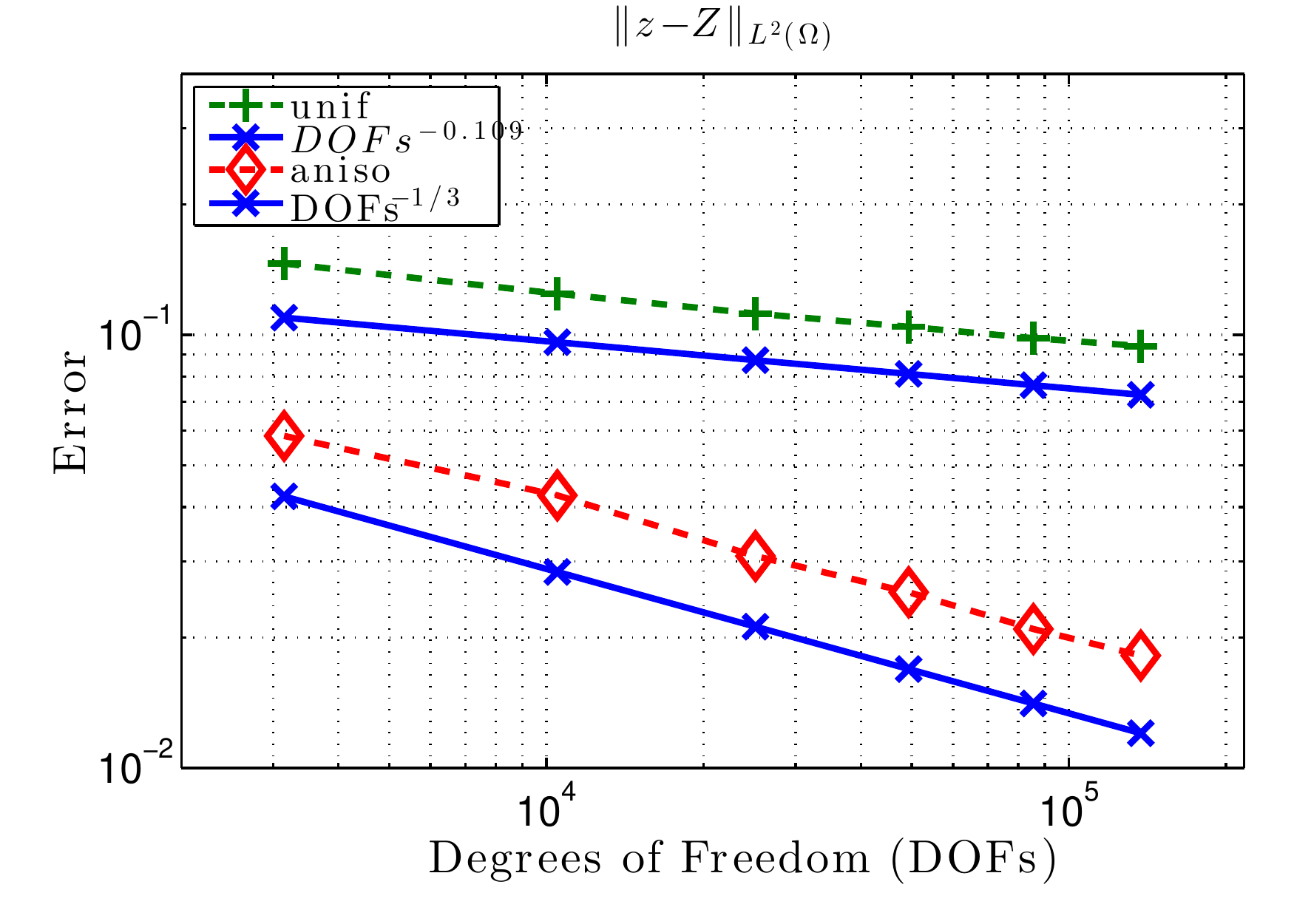}
\includegraphics[width=0.48\textwidth]{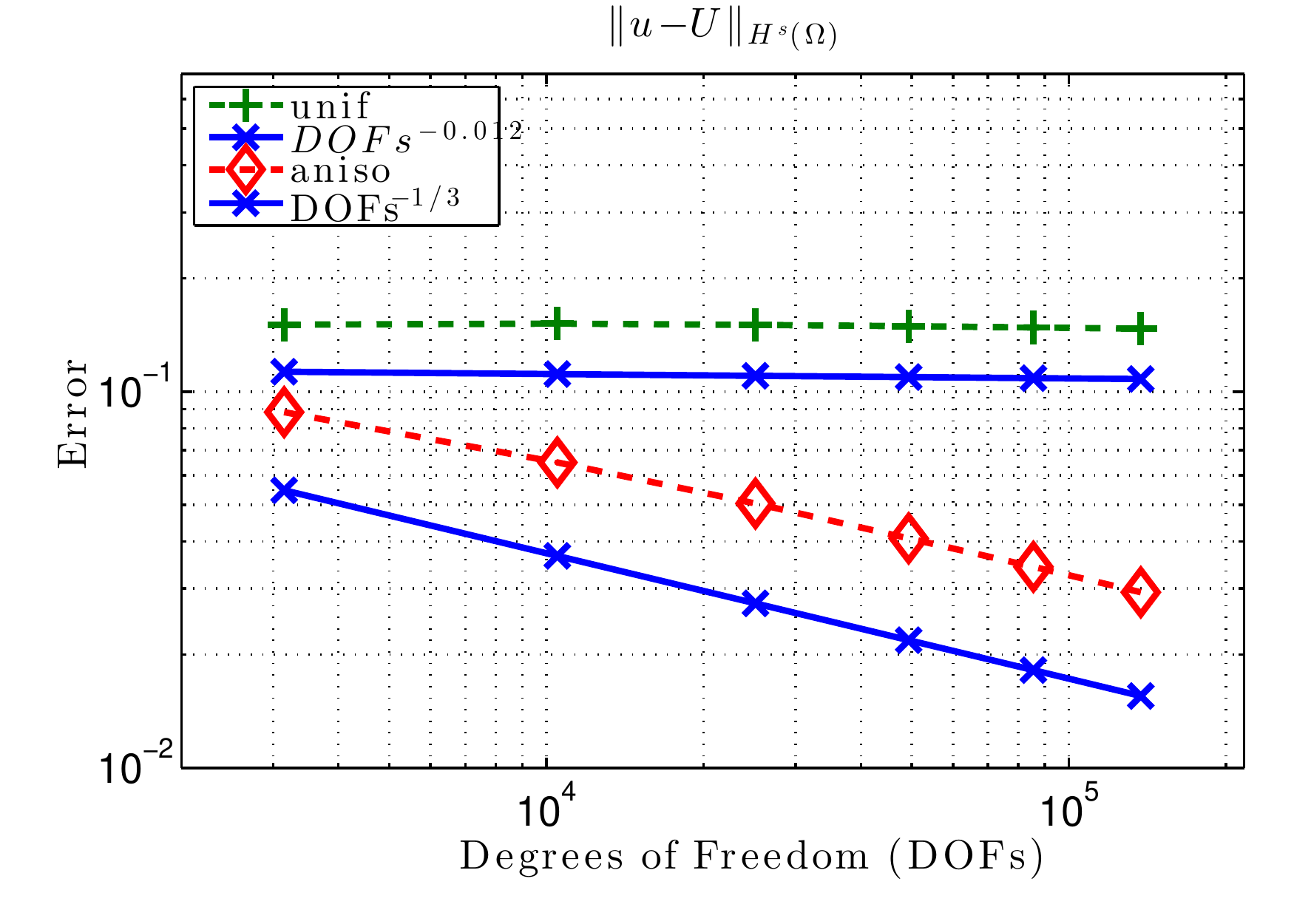}
\caption{\label{f:unif_ani_s005}
Computational rates of convergence for both quasi-uniform and anisotropic mesh refinements for $s = 0.5$. The left panel shows the corresponding rates for the control and the right one for the state. The Computational convergence rates on anisotropic meshes are in agreement with the estimates of Corollary~\ref{CR:fd}.}
\end{figure}

\subsection{Anisotropic refinement}
\label{s:grad}
The asymptotic relations 
$$\|\ozsf - \bar{Z} \|_{L^2(\Omega)} \approx (\# \T_{\Y})^{-\frac{1}{3}}, \qquad
\|\nabla(\oue - \bar{V}) \|_{L^2(y^{\alpha},\C)} \approx (\# \T_{\Y})^{-\frac{1}{3}}$$ 
are shown in Figure~\ref{f:grad_rate_s02} which illustrate the quasi-optimal decay rate of our fully-discrete scheme of \S\ref{subsec:fd} for all choices of the parameter $s$ considered. These examples show that anisotropy in the extended dimension is essential to recover optimality. We also present $L^2$-error estimates for the state variable, which decay as $(\# \T_{\Y})^{-\frac{2}{3}}$. The latter is not discussed in this paper and is part of a future work.

\begin{figure}[h!]
\centering
\includegraphics[width=0.44\textwidth]{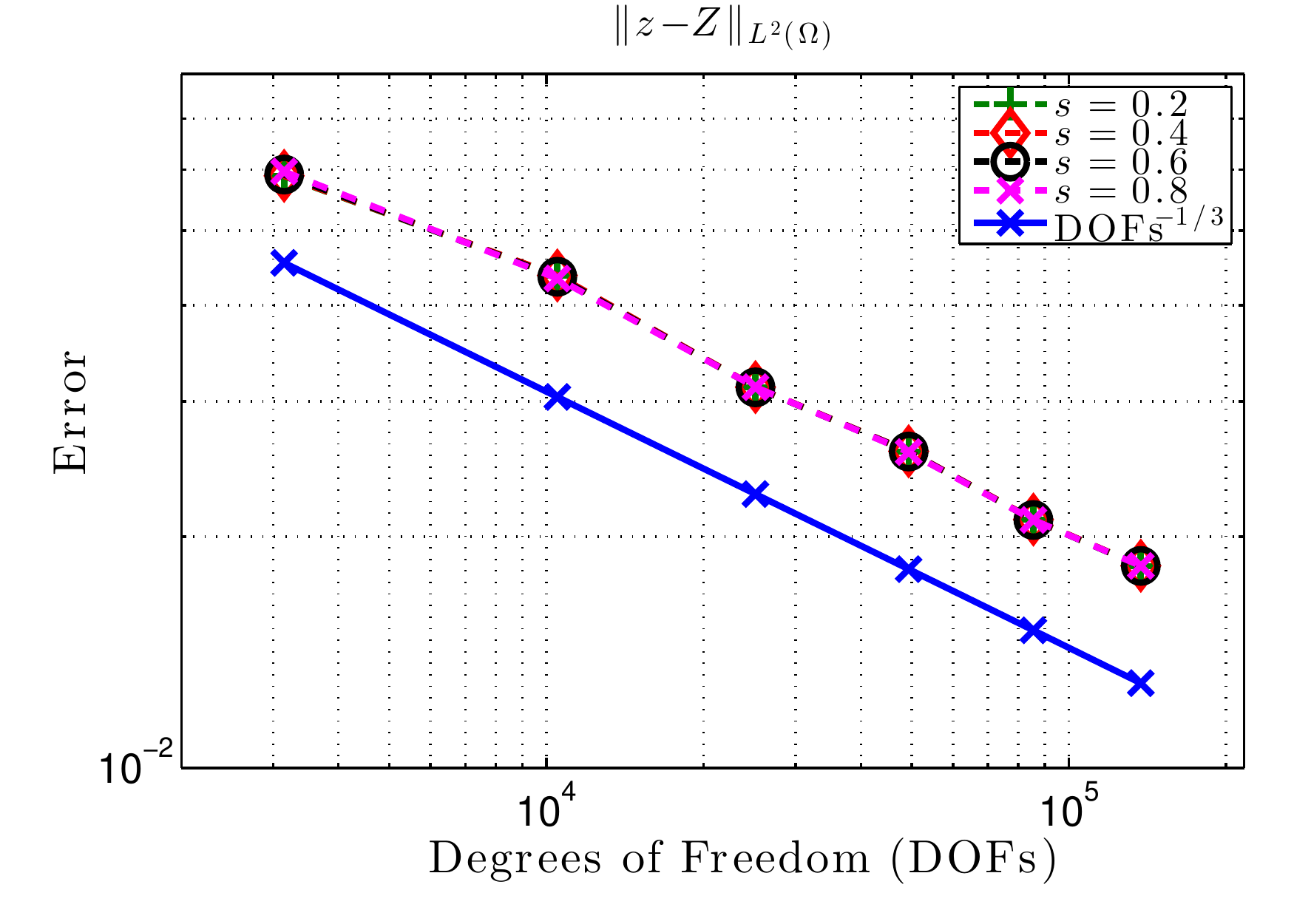} 
\includegraphics[width=0.44\textwidth]{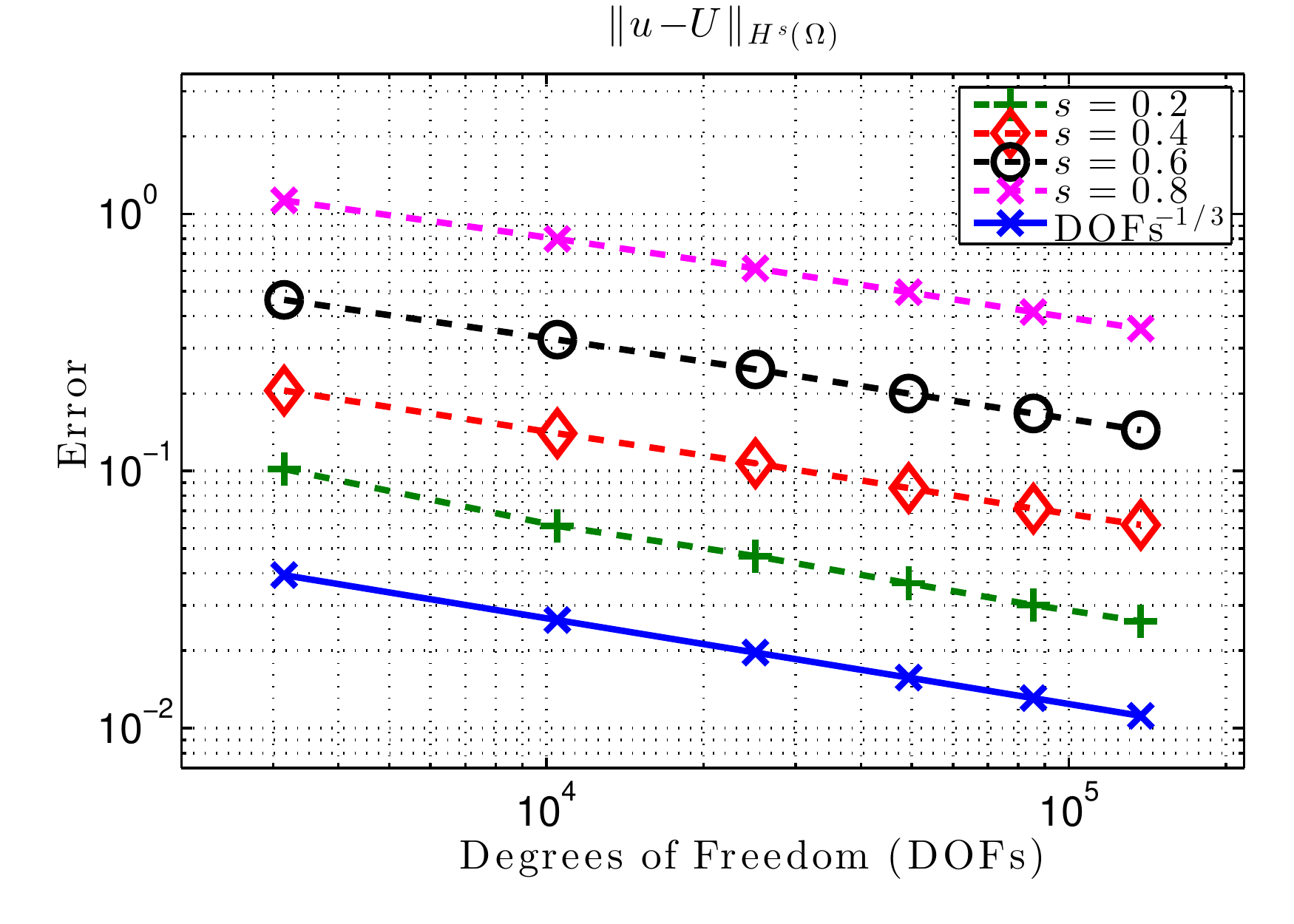}
\includegraphics[width=0.44\textwidth]{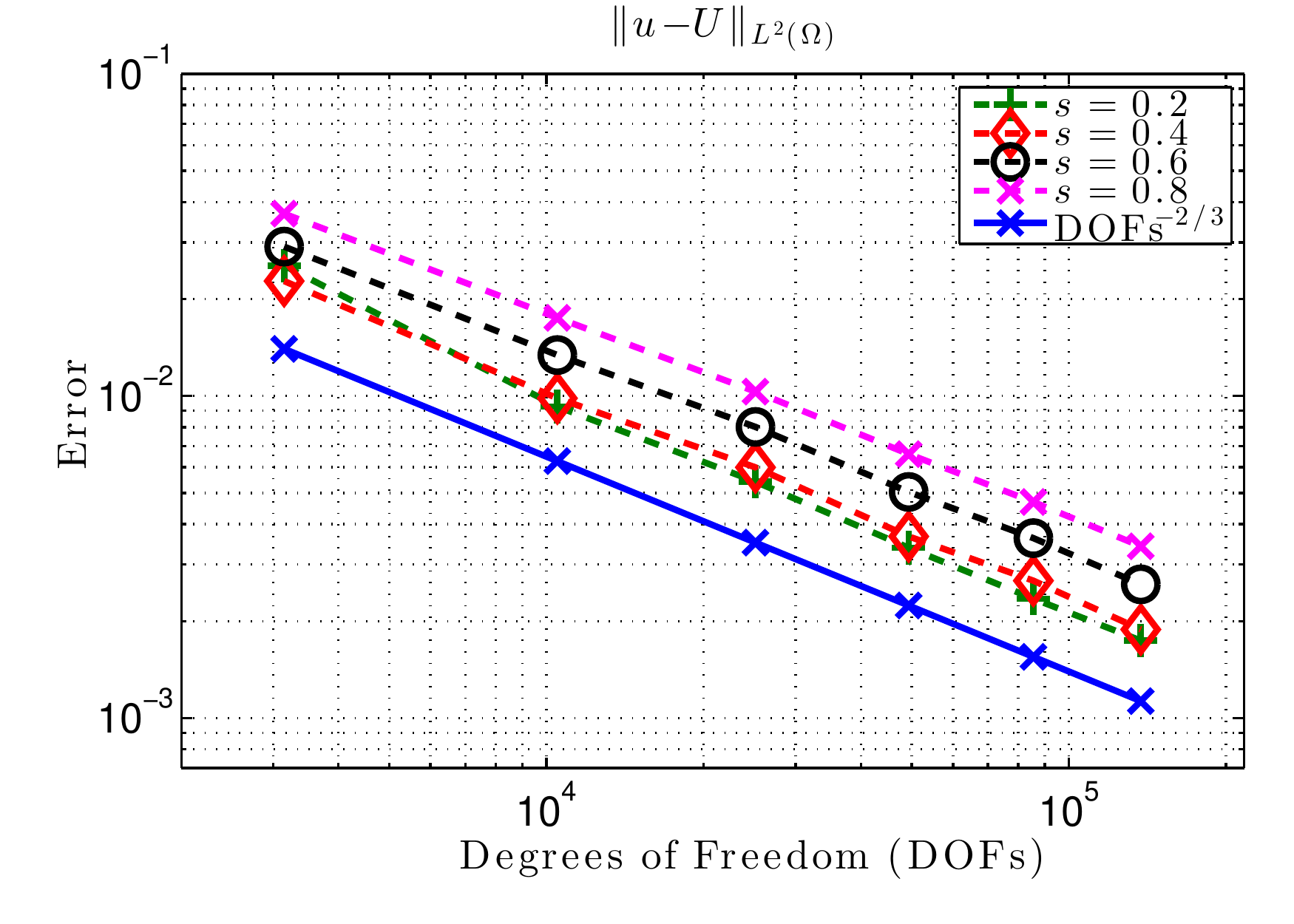}
\caption{\label{f:grad_rate_s02}
Computational rates of convergence for the fully-discrete scheme proposed in \S\ref{subsec:fd} on anisotropic meshes for
$n = 2$ and $s = 0.2$, $0.4$, $0.6$ and $s = 0.8$. The top left panel shows the decrease of the $L^2$-control error with respect to $\# \T_{\Y}$ and the top right the one for the $\Hs$-state error. In all cases we recover the rate $(\# \T_{\Y})^{-1/3}$. The bottom panel shows the rates of convergence  $(\# \T_{\Y})^{-2/3}$ for the $L^2$-state error. The
latter is not discussed in this paper.}
\end{figure}

\section*{Acknowledgment}
We would like to thank the referees for their insightful comments and for pointing out an inaccuracy in an earlier 
version of this work. We also wish to thank A. J. Salgado and B. Vexler for several fruitful discussions.

\bibliographystyle{plain}
\bibliography{biblio}

\def\cprime{$'$} \def\cprime{$'$} \def\cprime{$'$} \def\cprime{$'$}
  \def\cprime{$'$} \def\cprime{$'$}
\begin{thebibliography}{10}

\bibitem{Abra}
M.~Abramowitz and I.A. Stegun.
\newblock {\em Handbook of mathematical functions with formulas, graphs, and
  mathematical tables}, volume~55 of {\em National Bureau of Standards Applied
  Mathematics Series}.
\newblock 1964.

\bibitem{AOS}
A.~Antil, E.~Ot\'arola, and A.~J Salgado.
\newblock A fractional space-time optimal control problem: analysis and
  discretization.
\newblock arXiv:1504.00063, 2015.

\bibitem{HAntil_RHNochetto_PSodre_2014b}
H.~Antil, R.~H. Nochetto, and P.~Sodr\'e.
\newblock Optimal control of a free boundary problem with surface tension
  effects: A priori error analysis.
\newblock arXiv:1402.5709, 2014.

\bibitem{APR:12}
T.~Apel, J.~Pfefferer, and A.~R{\"o}sch.
\newblock Finite element error estimates for {N}eumann boundary control
  problems on graded meshes.
\newblock {\em Comput. Optim. Appl.}, 52(1):3--28, 2012.

\bibitem{ARD:09}
T.~Apel, A.~R{\"o}sch, and D.~Sirch.
\newblock {$L^\infty$}-error estimates on graded meshes with application to
  optimal control.
\newblock {\em SIAM J. Control Optim.}, 48(3):1771--1796, 2009.

\bibitem{ARW:07}
T.~Apel, A.~R{\"o}sch, and G.~Winkler.
\newblock Optimal control in non-convex domains: a priori discretization error
  estimates.
\newblock {\em Calcolo}, 44(3):137--158, 2007.

\bibitem{ACT:02}
N.~Arada, E.~Casas, and F.~Tr{\"o}ltzsch.
\newblock Error estimates for the numerical approximation of a semilinear
  elliptic control problem.
\newblock {\em Comput. Optim. Appl.}, 23(2):201--229, 2002.

\bibitem{atanackovic2014fractional}
T.M. Atanackovic, S.~Pilipovic, B.~Stankovic, and D.~Zorica.
\newblock {\em Fractional Calculus with Applications in Mechanics: Vibrations
  and Diffusion Processes}.
\newblock John Wiley \& Sons, 2014.

\bibitem{Bernardi}
C.~Bernardi.
\newblock Optimal finite-element interpolation on curved domains.
\newblock {\em SIAM J. Numer. Anal.}, 26(5):1212--1240, 1989.

\bibitem{BSV}
M~Bonforte, Y~Sire, and J.L. V\'azquez.
\newblock Existence, uniqueness and asymptotic behaviour for fractional porous
  medium equations on bounded domains.
\newblock arXiv:1404.6195, 2014.

\bibitem{bio}
A.~Bueno-Orovio, D.~Kay, V.~Grau, B.~Rodriguez, and K.~Burrage.
\newblock Fractional diffusion models of cardiac electrical propagation: role
  of structural heterogeneity in dispersion of repolarization.
\newblock {\em J. R. Soc. Interface}, 11(97), 2014.

\bibitem{CT:10}
X.~Cabr{\'e} and J.~Tan.
\newblock Positive solutions of nonlinear problems involving the square root of
  the {L}aplacian.
\newblock {\em Adv. Math.}, 224(5):2052--2093, 2010.

\bibitem{CS:07}
L.~Caffarelli and L.~Silvestre.
\newblock An extension problem related to the fractional {L}aplacian.
\newblock {\em Comm. Part. Diff. Eqs.}, 32(7-9):1245--1260, 2007.

\bibitem{CDDS:11}
A.~Capella, J.~D{\'a}vila, L.~Dupaigne, and Y.~Sire.
\newblock Regularity of radial extremal solutions for some non-local semilinear
  equations.
\newblock {\em Comm. Part. Diff. Eqs.}, 36(8):1353--1384, 2011.

\bibitem{CT:05}
E.~Casas, M.~Mateos, and F.~Tr{\"o}ltzsch.
\newblock Error estimates for the numerical approximation of boundary
  semilinear elliptic control problems.
\newblock {\em Comput. Optim. Appl.}, 31(2):193--219, 2005.

\bibitem{chen2009ifem}
L~Chen.
\newblock ifem: an integrated finite element methods package in matlab.
\newblock Technical report, Technical Report, University of California at
  Irvine, 2009.

\bibitem{CNOS}
L.~Chen, R.H. Nochetto, E.~Ot\'arola, and A.J. Salgado.
\newblock Multilevel methods for nonuniformly elliptic operators.
\newblock arXiv:1403.4278, 2014.

\bibitem{CNOS2}
L.~Chen, R.H. Nochetto, E.~Ot\'arola, and A.J. Salgado.
\newblock A pde approach to fractional diffusion: a posteriori error analysis.
\newblock {\em J. Comput. Phys.}, 2015.

\bibitem{wow}
W.~Chen.
\newblock A speculative study of $2/3$-order fractional laplacian modeling of
  turbulence: Some thoughts and conjectures.
\newblock {\em Chaos}, 16(2):1--11, 2006.

\bibitem{CiarletBook}
P.G. Ciarlet.
\newblock {\em The finite element method for elliptic problems}, volume~40 of
  {\em Classics in Applied Mathematics}.
\newblock SIAM, Philadelphia, PA, 2002.

\bibitem{DL:05}
R.G. Dur{\'a}n and A.L. Lombardi.
\newblock Error estimates on anisotropic {$Q_1$} elements for functions in
  weighted {S}obolev spaces.
\newblock {\em Math. Comp.}, 74(252):1679--1706 (electronic), 2005.

\bibitem{Guermond-Ern}
A.~Ern and J.-L. Guermond.
\newblock {\em Theory and practice of finite elements}, volume 159 of {\em
  Applied Mathematical Sciences}.
\newblock Springer-Verlag, New York, 2004.

\bibitem{GH:14}
P.~Gatto and J.~Hesthaven.
\newblock Numerical approximation of the fractional laplacian via hp-finite
  elements, with an application to image denoising.
\newblock {\em J. Sci. Comp.}, pages 1--22, 2014.

\bibitem{GU}
V.~Gol{\cprime}dshtein and A.~Ukhlov.
\newblock Weighted {S}obolev spaces and embedding theorems.
\newblock {\em Trans. Amer. Math. Soc.}, 361(7):3829--3850, 2009.

\bibitem{Grisvard}
P.~Grisvard.
\newblock {\em Elliptic problems in nonsmooth domains}, volume~24 of {\em
  Monographs and Studies in Mathematics}.
\newblock Pitman (Advanced Publishing Program), Boston, MA, 1985.

\bibitem{HB:10}
Y.~Ha and F.~Bobaru.
\newblock Studies of dynamic crack propagation and crack branching with
  peridynamics.
\newblock {\em Int. J. Fracture}, 162(1-2):229--244, 2010.

\bibitem{Hinze:05}
M.~Hinze.
\newblock A variational discretization concept in control constrained
  optimization: the linear-quadratic case.
\newblock {\em Comput. Optim. Appl.}, 30(1):45--61, 2005.

\bibitem{HPUU:09}
M.~Hinze, R.~Pinnau, M.~Ulbrich, and S.~Ulbrich.
\newblock {\em Optimization with {PDE} constraints}, volume~23 of {\em
  Mathematical Modelling: Theory and Applications}.
\newblock Springer, New York, 2009.

\bibitem{HT:10}
M.~Hinze and F.~Tr{\"o}ltzsch.
\newblock Discrete concepts versus error analysis in {PDE}-constrained
  optimization.
\newblock {\em GAMM-Mitt.}, 33(2):148--162, 2010.

\bibitem{HO}
Y.~Huang and A.~Oberman.
\newblock Numerical methods for the fractional laplacian: A finite
  difference-quadrature approach.
\newblock {\em SIAM J. Numer. Anal.}, 52(6):3056--3084, 2014.

\bibitem{ICH}
R.~Ishizuka, S.-H. Chong, and F.~Hirata.
\newblock An integral equation theory for inhomogeneous molecular fluids: The
  reference interaction site model approach.
\newblock {\em J. Chem. Phys}, 128(3), 2008.

\bibitem{IK:08}
K.~Ito and K.~Kunisch.
\newblock {\em Lagrange multiplier approach to variational problems and
  applications}, volume~15 of {\em Advances in Design and Control}.
\newblock Society for Industrial and Applied Mathematics (SIAM), Philadelphia,
  PA, 2008.

\bibitem{KSbook}
G.~Kinderlehrer, D.and~Stampacchia.
\newblock {\em An introduction to variational inequalities and their
  applications}, volume~88 of {\em Pure and Applied Mathematics}.
\newblock Academic Press, Inc. [Harcourt Brace Jovanovich, Publishers], New
  York-London, 1980.

\bibitem{KO84}
A.~Kufner and B.~Opic.
\newblock How to define reasonably weighted {S}obolev spaces.
\newblock {\em Comment. Math. Univ. Carolin.}, 25(3):537--554, 1984.

\bibitem{MR2064019}
S.~Z. Levendorski{\u\i}.
\newblock Pricing of the {A}merican put under {L}\'evy processes.
\newblock {\em Int. J. Theor. Appl. Finance}, 7(3):303--335, 2004.

\bibitem{Lions}
J.-L. Lions and E.~Magenes.
\newblock {\em Non-homogeneous boundary value problems and applications. {V}ol.
  {I}}.
\newblock Springer-Verlag, New York, 1972.

\bibitem{MV:08}
D.~Meidner and B.~Vexler.
\newblock A priori error estimates for space-time finite element discretization
  of parabolic optimal control problems part ii: Problems with control
  constraints.
\newblock {\em SIAM J. Control Optim.}, 47(3):1301--1329, 2008.

\bibitem{NOS}
R.~H. Nochetto, E.~Ot\'arola, and A.~J. Salgado.
\newblock A pde approach to fractional diffusion in general domains: A priori
  error analysis.
\newblock {\em Found. Comput. Math.}, pages 1--59, 2014.

\bibitem{NOS3}
R.H. Nochetto, E.~Ot\'arola, and A.J. Salgado.
\newblock A {PDE} approach to space-time fractional parabolic problems.
\newblock arXiv:1404.0068, 2014.

\bibitem{NOS2}
R.H. Nochetto, E.~Ot\'arola, and A.J. Salgado.
\newblock Piecewise polynomial interpolation in muckenhoupt weighted sobolev
  spaces and applications.
\newblock {\em Numer. Math.}, pages 1--46, 2015.

\bibitem{R:06}
A.~R{\"o}sch.
\newblock Error estimates for linear-quadratic control problems with control
  constraints.
\newblock {\em Optim. Methods Softw.}, 21(1):121--134, 2006.

\bibitem{ST:10}
P.R. Stinga and J.L. Torrea.
\newblock Extension problem and {H}arnack's inequality for some fractional
  operators.
\newblock {\em Comm. Part. Diff. Eqs.}, 35(11):2092--2122, 2010.

\bibitem{Tartar}
L.~Tartar.
\newblock {\em An introduction to {S}obolev spaces and interpolation spaces},
  volume~3 of {\em Lecture Notes of the Unione Matematica Italiana}.
\newblock Springer, Berlin, 2007.

\bibitem{Tbook}
F.~Tr{\"o}ltzsch.
\newblock {\em Optimal Control of Partial Differential Equations: Theory,
  Methods, and Applications}.
\newblock Graduate Studies in Mathematics. American Mathematical Society, 2010.

\bibitem{Turesson}
B.O. Turesson.
\newblock {\em Nonlinear potential theory and weighted {S}obolev spaces},
  volume 1736 of {\em Lecture Notes in Mathematics}.
\newblock Springer-Verlag, Berlin, 2000.

\end{thebibliography}

\end{document}